\documentclass[12pt,fleqn,leqno]{amsart}
\usepackage{amssymb,amsthm}
\usepackage{mathrsfs}  
\usepackage{enumitem}
\usepackage[bb=boondox]{mathalfa}
\usepackage[svgnames]{xcolor}
\usepackage[colorlinks,linkcolor=blue,citecolor=Green]{hyperref}
\usepackage{titlesec,titletoc,chngcntr}

\newtheorem{theorem}{Theorem}[section]
\newtheorem{lemma}[theorem]{Lemma}
\newtheorem{corollary}[theorem]{Corollary}
\newtheorem{proposition}[theorem]{Proposition}

\theoremstyle{definition}
\newtheorem{question}{Question}
\newtheorem{example}[theorem]{Example}
\newtheorem{remark}[theorem]{Remark}

\renewcommand{\emptyset}{\varnothing}

\DeclareMathOperator{\uhr}{\upharpoonright}  
\DeclareMathOperator{\N}{\mathbb{N}}
\DeclareMathOperator{\R}{\mathbb{R}}
\DeclareMathOperator{\B}{\mathbb{B}}
\DeclareMathOperator{\s}{\mathbb{S}}
\DeclareMathOperator{\D}{\mathbb{D}}
\DeclareMathOperator{\sto}{\leadsto}
\DeclareMathOperator{\conv}{conv}
\DeclareMathOperator{\ldim}{loc\hspace{1pt}dim}
\DeclareMathOperator{\diam}{diam}

\setlist{noitemsep}
\setenumerate{labelindent=\parindent,label=\upshape{(\alph*)}}

\numberwithin{equation}{section} 

\titleformat{\section}[block]
 {\bfseries}
 {\thesection.}
 {\fontdimen2\font}
 {}

\newcommand{\periodafter}[1]{#1.}

\titleformat{\subsection}[runin]
 {\bfseries}
 {\thesubsection.}
 {\fontdimen2\font}
 {\periodafter}

\titlecontents{section}[1.5em]
 {}
 {\makebox[1.5em][l]{\thecontentslabel.}}
 {}
 {\titlerule*[6pt]{.}\contentspage\hspace*{2.5em}}

\titlecontents*{subsection}[3.8em]
{\small}{\thecontentslabel. }{}
{ (\thecontentspage)}[  ][\,\ $\diamond$\,\ ][.\vspace{5pt}]

\begin{document}

 \author{Valentin Gutev}

 \address{Department of Mathematics, Faculty of Science, University of
   Malta, Msida MSD 2080, Malta} \email{valentin.gutev@um.edu.mt}



 \title{Selections and Higher Separation Axioms}

 \begin{abstract}
   This survey presents some historical background and recent
   developments in the area of selections for set-valued mappings
   along with several open questions. It was written with the hope
   that the presented material may pique an interest in the selection
   problem for set-valued mappings --- a problem with a fascinating
   history and appealing applications.
\end{abstract}

\date{\today}
\maketitle

\startcontents[sections]
\vspace*{1pc}
 \printcontents[sections]{l}{1}{\setcounter{tocdepth}{2}}
\vspace*{1pc}

\section{The Selection Problem}

In the mid 1950's Ernest Michael wrote a series of fundamental papers
\cite{michael:56a,michael:56b,michael:57} relating familiar extension
theorems to selections, thus laying down the foundation of the theory
of continuous selections. Nowadays, selections became an indispensable
tool for many mathematicians working in vastly different
areas. However, the key importance of Michael's selection theory is
not only in providing a comprehensive solution to diverse selection
problems, but also in the immediate inclusion of the obtained results
into the general context of development of topology. In the first of
these papers \cite{michael:56a}, Michael wrote:\medskip

``One of the most interesting and important problems in topology is
the \emph{extension problem}: Two topological spaces $X$ and $Y$ are
given, together with a closed $A\subset X$, and we would like to know
whether every continuous function $g:A\to Y$ can be extended to a
continuous function $f$ from $X$ (or at least from some open $U\supset
A$) into $Y$. Sometimes there are additional requirements on $f$,
which frequently (as in the theory of fibre bundles) take the
following form: For every $x$ in $X$, $f(x)$ must be an element of a
pre-assigned subset of $Y$. This new problem, which we call the
\emph{selection problem}, is clearly more general than the extension
problem, and presents a challenge even when $A$ is the null set or a
1-point set (where the extension problem is trivial).''\medskip

Michael proceeded to remark that with the exception of sandwich-like
theorems, such as Tong \cite{tong:48,MR0050265}, Kat\v{e}tov
\cite{katetov:51,MR0060211} and Dowker \cite{dowker:51}, no attempt
has been made to obtain results under minimal hypotheses. Then he
stated the purpose of his papers emphasising that most of the familiar
extension theorems, such as the Tietze-Urysohn extension theorem,
Kuratowski's extension theorem for finite-dimensional spaces and the
Borsuk homotopy extension theorem can be slightly altered and
essentially generalised in terms of selection theorems.\medskip

The present survey contains different aspects of Michael's selection
theory, possible new research directions along with several open
questions.

\section{Selections and Dowker's Extension Theorem}

In this section, and the rest of the paper, all spaces are Hausdorff
topological spaces. However, several of the considerations are also
valid for $T_1$-spaces, and some even without assuming any separation
axioms.

\subsection{Collectionwise Normality}

A space $X$ is \emph{collectionwise normal} (Bing \cite{bing:51}) if
for every discrete collection $\mathscr{D}$ of subsets of $X$ there
exists a discrete collection $\{U_D:D\in \mathscr{D}\}$ of open
subsets of $X$ such that ${D\subset U_D}$, for all $D\in
\mathscr{D}$. In this case, we simply say that $\mathscr{D}$ has an
\emph{open discrete expansion}.  Every collectionwise normal space is
normal, but the converse is not necessarily true.  In the same paper,
see \cite[Examples G and H]{bing:51}, Bing described an example of a
normal space which is not collectionwise normal, now known as Bing's
example \cite[5.1.23 Bing's Example]{engelking:89}. He also proved
that full normality (i.e.\ paracompactness) implies collectionwise
normality but not conversely \cite[Theorem 12]{bing:51}. Recall that a
space $X$ is \emph{paracompact} if every open cover of $X$ has a
locally finite open refinement.\medskip

Since the closure of the elements of any discrete collection is a
discrete collection too, a space $X$ is collectionwise normal if
every discrete collection $\mathscr{D}$ of closed subsets of $X$ has
an open discrete expansion. In the presence of normality, this can be
further relaxed to requiring $\mathscr{D}$ to have only an \emph{open
  pairwise-disjoint expansion} (i.e.\ an open pairwise-disjoint family
$\{U_D:D\in \mathscr{D}\}$ with ${D\subset U_D}$, for all
$D\in \mathscr{D}$).\medskip

Collectionwise normality is a natural generalisation of
normality. Indeed, if $X$ has this property with respect to discrete
families $\mathscr{D}$ of cardinality at most $\tau$, then it is
called \emph{$\tau$-collectionwise normal}. Thus, $X$ is normal iff it
is $2$-collectionwise normal (or, more generally, $n$-collectionwise
normal for every finite $n\geq 2$). In fact, a space $X$ is normal iff
it is $\omega$-collectionwise normal, which follows easily from
Urysohn's characterisation of normality. On the other hand, for every
infinite cardinal $\tau$ there exists a $\tau$-col\-lectionwise normal
space which is not $\tau^+$-collectionwise normal
\cite{przymusinski:78a}, where the cardinal $\tau^+$ is the immediate
successor of~$\tau$.\medskip

Collectionwise normality is the natural domain for continuous
extensions. The following extension theorem was proved by Dowker
\cite{dowker:52}, and is commonly called \emph{Dowker's extension
  theorem}.

\begin{theorem}[\cite{dowker:52}]
  \label{theorem-res-pro-v1:1}
  A space $X$ is collectionwise normal if and only if for every
  closed subset $A\subset X$, every continuous map from $A$ to a
  Banach space $E$ can be continuously extended to the whole of $X$.
\end{theorem}

For a simple proof of Theorem \ref{theorem-res-pro-v1:1}, the
interested reader is referred to \cite{MR3673071}.

\subsection{Selections and Extensions}

For a space $Y$, let $\mathscr{F}(Y)$ be the collection of
all nonempty closed subsets of $Y$; $\mathscr{C}(Y)$ --- that of all
compact members of $\mathscr{F}(Y)$; and
$\mathscr{C}'(Y)=\mathscr{C}(Y)\cup\{Y\}$. For a (not necessarily
convex) subset $Y$ of a normed space $E$, we will use the subscript
``$\mathbf{c}$'' to denote the convex members of anyone of the above
families; namely, $\mathscr{F}_\mathbf{c}(Y)$ for the convex members
of $\mathscr{F}(Y)$; $\mathscr{C}_\mathbf{c}(Y)$ for those of
$\mathscr{C}(Y)$; and
$\mathscr{C}_\mathbf{c}'(Y)=\mathscr{C}_\mathbf{c}(Y)\cup
\{Y\}$. \medskip

For spaces (sets) $X$ and $Y$, we write $\Phi:X\sto Y$ to designate
that $\Phi$ is a \emph{set-valued} (or \emph{multi-valued})
\emph{mapping} from $X$ to the nonempty subsets of $Y$. Such a mapping
$\Phi:X\sto Y$ is \emph{lower semi-continuous}, or \emph{l.s.c.}, if
the set
\[
  \Phi^{-1}[U]=\{x\in X: \Phi(x)\cap U\neq \emptyset\}
\]
is open in $X$, for every open $U\subset Y$. Finally, let us recall
that a map $f:X\to Y$ is a \emph{selection} (or a \emph{single-valued
  selection}) for $\Phi:X\sto Y$ if $f(x)\in\Phi(x)$, for every
$x\in X$. The following natural relationship between selections and
extensions is well known.

\begin{proposition}
  \label{proposition-contr-ext-14:1}
  For a subset $A\subset X$ and $g:A\to Y$, define a set-valued
  mapping $\Phi_g:X\to \mathscr{C}'(Y)$ by $\Phi_g(x)=\{g(x)\}$ if
  $x\in A$, and $\Phi_g(x)=Y$ otherwise. Then
  \begin{enumerate}[label=\upshape{(\roman*)}]
  \item\label{item:4} $f:X\to Y$ is a continuous selection for
    $\Phi_g$ if and only if it is a continuous extension for $g$.
  \item\label{item:3} If\/ $Y$ has at least two points, then $\Phi_g$
    is l.s.c.\ if and only if $A$ is closed and $g$ is continuous.
  \end{enumerate}
\end{proposition}

Proposition \ref{proposition-contr-ext-14:1}\ref{item:4} is
\cite[Example 1.3]{michael:56a}, and follows immediately from the
definition of $\Phi_g$. The one direction
of Proposition \ref{proposition-contr-ext-14:1}\ref{item:3} is
\cite[Example 1.3*]{michael:56a}; the other follows from the
fact that $Y$ contains nonempty disjoint open sets $U$ and $V$, hence
$X\setminus A=\Phi_g^{-1}[U]\cap \Phi_g^{-1}[V]$.\medskip

According to Proposition \ref{proposition-contr-ext-14:1}, any
extension theorem can be restated as a selection theorem for the
mapping $\Phi_g$. Based on this, Michael
\cite{michael:56a,michael:56b,michael:57} proposed several selection
theorems generalising ordinary extension theorems. Here are two of
them whose prototype is Dowker's extension theorem (Theorem
\ref{theorem-res-pro-v1:1}).

\begin{theorem}[\cite{michael:56a}]
  \label{theorem-res-pro-v1:2}
  A space $X$ is collectionwise normal if and only if for every
  Banach space $E$, every l.s.c.\ mapping $\Phi:X\to
  \mathscr{C}'_\mathbf{c}(E)$ has a continuous selection.
\end{theorem}

\begin{theorem}[\cite{michael:56a}]
  \label{theorem-res-pro-v9:1}
  A space $X$ is paracompact if and only if for every Banach space
  $E$, every l.s.c.\ mapping $\Phi:X\to \mathscr{F}_\mathbf{c}(E)$
  has a continuous selection.
\end{theorem}

However, the proof of Theorem \ref{theorem-res-pro-v1:2} in
\cite{michael:56a} was incomplete and, in fact, working only for the
case of $\mathscr{C}_\mathbf{c}(E)$-valued mappings. The first
complete proof of this theorem was given by Choban and Valov
\cite{choban-valov:75} using a different technique. A simple proof of
Theorem \ref{theorem-res-pro-v1:2}, based on Theorem
\ref{theorem-res-pro-v1:1}, was given in \cite{MR2643824}.

\subsection{PF-Normality}

For an infinite cardinal number $\tau$, a space $X$ is called
\emph{$\tau$-PF-normal} (see \cite{smith:72}) if every point-finite
open cover of $X$ of cardinality $\leq \tau$ is normal.  Every
$\tau$-collectionwise normal space is $\tau$-PF-normal (see
\cite{michael:55}), and $\omega$-PF-normality coincides with normality
\cite{morita:48}.  However, PF-normality is neither identical to
collectionwise normality (see Bing's example \cite{bing:51} and
\cite[Example 1]{michael:55}), nor to normality (\cite[Example
2]{michael:55}). For some properties of PF-normal spaces, the
interested reader is referred to \cite[Section
3]{gutev-ohta-yamazaki:03} and \cite{michael:55}.\medskip

In the realm of normal spaces, PF-normal spaces coincide with the
\emph{point-finitely paracompact spaces} in the sense of Kand\^o
\cite{MR0063648}; while, in Nedev's terminology \cite{nedev:80},
$\tau$-PF-normal spaces are precisely the
\emph{$\tau$-pointwise-א$\aleph_0$-paracompact spaces}. The following
characterisation of PF-normal spaces was obtained by Kand\^o
\cite{MR0063648} and Nedev \cite{nedev:80}.

\begin{theorem}[\cite{MR0063648,nedev:80}]
  \label{theorem-res-pro-v1:3}
  A space $X$ is PF-normal if and only if for every normed space $E$,
  every l.s.c.\ mapping $\varphi:X\to \mathscr{C}_\mathbf{c}(E)$ has a
  continuous selection.
\end{theorem}

\subsection{Selections and Collectionwise Normality}

For a metric space $(Y,\rho)$ and $\varepsilon>0$, we will use
$\mathbf{O}_\varepsilon(p)$ for the open $\varepsilon$-ball centred at
a point $p\in Y$; and
$\mathbf{O}_\varepsilon(S)=\bigcup_{p\in S}\mathbf{O}_\varepsilon(p)$,
whenever $S\subset Y$. A map $f:X\to Y$ is an
\emph{$\varepsilon$-selection} for a mapping $\Phi :X\sto Y$ if
$f(x)\in \mathbf{O}_\varepsilon\left(\Phi (x)\right)$, for every
$x\in X$.\medskip

Although Theorem \ref{theorem-res-pro-v1:2} is valid in the way it was
stated by Michael, it was shown in \cite[Theorem 1.2]{MR2643824} that
it is equivalent to both Theorem \ref{theorem-res-pro-v1:1} and
Theorem \ref{theorem-res-pro-v1:3}. In fact, a bit more was obtained
in \cite{MR2643824}, see \cite[Claim 2.1]{MR2643824}.

\begin{theorem}[\cite{MR2643824}]
  \label{theorem-res-pro-v1:4}
  Let $E$ be a Banach space. Then for a space $X$, the following are
  equivalent\textup{:}
  \begin{enumerate}
  \item\label{item:8} Every l.s.c.\ mapping $\Phi:X\to
    \mathscr{C}'_\mathbf{c}(E)$ has a continuous selection.
  \item\label{item:9} If $A\subset X$ is closed, then every l.s.c.\
    mapping $\varphi : A\to \mathscr{C}_\mathbf{c}(E)$ has a
    continuous selection, and every continuous $g:A\to E$ can be
    extended to a continuous map $f:X\to E$.
  \item \label{item:shsa-vgg-rev:1} If\/ $\Phi:X\sto E$ is an l.s.c.\
    convex-valued mapping which admits a continuous map $g:X\to E$
    such that $\Phi(x)$ is compact whenever $g(x)\notin \Phi(x)$,
    then $\Phi$ has a continuous $\varepsilon$-selection, for every
    $\varepsilon>0$.
  \end{enumerate}
\end{theorem}

It should be remarked that the proof of Theorem
\ref{theorem-res-pro-v1:4} is straightforward avoiding any explicit
reference to collectionwise normality. It should be also remarked that
the equivalence \ref{item:8}$\ \Leftrightarrow\ $\ref{item:9} in Theorem
\ref{theorem-res-pro-v1:4} represents the selection property of
PF-normality (Theorem \ref{theorem-res-pro-v1:3}) which, together with
Dowker's extension theorem (Theorem \ref{theorem-res-pro-v1:1}), gives
Theorem \ref{theorem-res-pro-v1:2}.  This brings the natural question
of whether there is any selection generalisation of Dowker's extension
theorem.  Various aspects of this question are discussed below.

\subsection{Selection Compact-Deficiency}

Motivated by \ref{item:shsa-vgg-rev:1} of Theorem
\ref{theorem-res-pro-v1:4}, the following question was posed in
\cite[Question 1]{MR2643824}.

\begin{question}[\cite{MR2643824}]
  \label{question-res-pro-v1:1}
  Let $X$ be a collectionwise normal space and $E$ be a Banach
  space. Suppose that $\Phi:X\to \mathscr{F}_\mathbf{c}(E)$ is an
  l.s.c.\ mapping which admits a continuous map $g:X\to E$ with
  $\Phi(x)$ compact, whenever $g(x)\notin \Phi(x)$ for some $x\in
  X$. Does $\Phi$ have a continuous selection?
\end{question}

If $\Phi:X\to \mathscr{C}'(E)$ and $g:X\to E$ is any map, then
$\Phi(x)$ is compact for every $x\in X$ with $g(x)\notin
\Phi(x)$. Moreover, for a collectionwise normal space $X$ and a Banach
space $E$, every l.s.c.\ mapping
$\Phi:X\to \mathscr{C}'_\mathbf{c}(E)$ has a continuous selection, by
Theorem \ref{theorem-res-pro-v1:2}. Thus, the answer to Question
\ref{question-res-pro-v1:1} is ``Yes'' for
$\mathscr{C}'_\mathbf{c}(E)$-valued mappings. A natural example of
mappings which are as in Question \ref{question-res-pro-v1:1}, but not
$\mathscr{C}'_\mathbf{c}(E)$-valued, is given at the end of this
section, see Remark \ref{remark-shsa-v18:1}.\medskip

If $\mathbf{0}\in E$ is the origin of a normed space $E$ and
$\varphi:X\to \mathscr{C}'(E)$, then $\varphi(x)$ is compact for every
$x\in X$ with $\mathbf{0}\notin \varphi(x)$. This property is
equivalent to the one in Question \ref{question-res-pro-v1:1}. Indeed,
let $\Phi:X\to \mathscr{F}_\mathbf{c}(E)$ and $g:X\to E$ be a
continuous map with $\Phi(x)$ compact, whenever $g(x)\notin \Phi(x)$
for some $x\in X$. Define a mapping
$\varphi:X\to \mathscr{F}_\mathbf{c}(E)$ by $\varphi(x)=\Phi(x)-g(x)$,
$x\in X$. Then $\varphi(x)$ is compact for every $x\in X$ with
$\mathbf{0}\notin \varphi(x)$, and $\varphi$ has a continuous
selection if and only if so does $\Phi$. Thus, we have the following
relaxed form of Question \ref{question-res-pro-v1:1}.

\begin{question}
  \label{question-shsa-v15:1}
  Let $X$ be a collectionwise normal space, $E$ be a Banach space and
  $\varphi:X\to \mathscr{F}_\mathbf{c}(E)$ be an l.s.c.\ mapping such
  that $\varphi(x)$ is compact, for every $x\in X$ with
  $\mathbf{0}\notin \varphi(x)$. Does $\varphi$ have a continuous
  selection?
\end{question}

Here are some further remarks regarding some of the challenges in this
question.

\begin{proposition}
  \label{proposition-res-pro-v11:1}
  Let $E$ be a metrizable space, $\varphi:X\to \mathscr{F}(E)$ be
  l.s.c.\ and $g:X\to E$ be continuous. Then
  $\left\{x\in X: g(x)\in \varphi(x)\right\}$ is a $G_\delta$-set in
  $X$.
\end{proposition}

Complementary to Proposition \ref{proposition-res-pro-v11:1} is the
following observation.

\begin{proposition}
  \label{proposition-res-pro-v11:2}
  Let $X$ be a collectionwise normal space, $E$ be a normed space and
  ${\varphi:X\to \mathscr{F}_\mathbf{c}(E)}$ be an l.s.c.\ mapping
  such that $\varphi(x)$ is compact, for every $x\in X$ with
  $\mathbf{0}\notin \varphi(x)$.  Set
  $H=\{x\in X: \mathbf{0}\notin \varphi(x)\}$. Then $\varphi\uhr H$
  has a continuous selection. If, moreover, $E$ is a Banach space,
  then $\varphi\uhr G$ has a continuous selection for some
  $G_\delta$-subset $G\subset X$ with $H\subset G$.
\end{proposition}

\begin{proof}
  According to \cite[Theorem 1.3]{MR0112115} and Proposition
  \ref{proposition-res-pro-v11:1}, $H$ is itself a collectionwise
  normal space. Hence, by Theorem \ref{theorem-res-pro-v1:2} (see also
  Theorem \ref{theorem-res-pro-v1:3}), $\varphi\uhr H$ has a
  continuous selection $h:H\to E$. If $E$ is a Banach space, then $h$
  can be extended to a continuous map $g:Z\to E$ for some
  $G_\delta$-subset $Z\subset X$ containing $H$, see \cite[Theorem
  4.3.20]{engelking:89}. Applying Proposition
  \ref{proposition-res-pro-v11:1} once more, the set
  $G=\{x\in Z: g(x)\in \varphi(x)\}$ is also $G_\delta$, and clearly
  contains $H$.
\end{proof}

Propositions \ref{proposition-res-pro-v11:1} and
\ref{proposition-res-pro-v11:2} show that the mapping
$\varphi:X\to \mathscr{F}_\mathbf{c}(E)$ in Question
\ref{question-shsa-v15:1} has two partial continuous selections on
complementary subsets of the domain. Hence, a particular challenge in
this question (and in Question \ref{question-res-pro-v1:1} as well) is
if one can use these partial selections to construct a continuous
selection for the mapping $\varphi$ itself. This brings the following
alternative question.

\begin{question}
  \label{question-shsa-v13:1}
  Let $X$ be a collectionwise normal space, $E$ be a Banach space,
  $\varphi:X\to \mathscr{F}_\mathbf{c}(E)$ be an l.s.c.\ mapping and
  $H=\{x\in X: \mathbf{0}\notin \varphi(x)\}$. Does there exist a
  continuous selection for $\varphi$ provided $\varphi\uhr H$ has a
  continuous selection?
\end{question}

The property in Question \ref{question-shsa-v13:1} implies
collectionwise normality of $X$. Indeed, take a closed set
$A\subset X$ and a continuous map $g:A\to E$ into a Banach space
$E$. Next, let $\Phi_g:X\to \mathscr{C}'_\mathbf{c}(E)$ be the
associated mapping defined as in Proposition
\ref{proposition-contr-ext-14:1}. Then $\Phi_g$ is l.s.c.\ and
$H=\{x\in X: \mathbf{0}\notin \Phi_g(x)\}\subset A$.  Hence, $g\uhr H$
is a continuous selection for $\Phi_g\uhr H$. Thus, by Proposition
\ref{proposition-contr-ext-14:1}, the existence of a continuous
selection for $\Phi_g$ (as per Question \ref{question-shsa-v13:1}) is
equivalent to the existence of a continuous extension of $g$, i.e.\ to
collectionwise normality of $X$ (by Theorem
\ref{theorem-res-pro-v1:1}). This suggests the following
interpretation of Question \ref{question-shsa-v13:1} in the setting of
PF-normal spaces.

\begin{question}
  \label{question-shsa-v13:2}
  Let ${\varphi:X\to \mathscr{F}_\mathbf{c}(E)}$ be an l.s.c.\
  mapping, where $X$ is a PF-normal space and $E$ is a Banach
  space. Suppose that there exists a continuous map $h:X\to E$ such
  that $h(x)\in \varphi(x)$, for every $x\in X$ with
  $\mathbf{0}\notin \varphi(x)$. Does $\varphi$ have a continuous
  selection?
\end{question}

\subsection{Approximate Selections}

In what follows, to each $\Phi:X\sto Y$ we will associate the mapping
$\overline{\Phi}:X\to \mathscr{F}(Y)$ defined by
$\overline{\Phi}(x)=\overline{\Phi(x)}$, $x\in X$. Moreover, for a
pair of mappings $\Phi,\Psi: X\sto Y$ with
$\Phi(x)\cap \Psi(x)\neq \emptyset$, $x\in X$, we will use
$\Phi\wedge\Psi$ to denote their intersection, i.e.\ the mapping which
assigns to each $x\in X$ the set
$[\Phi\wedge\Psi](x)=\Phi(x)\cap \Psi(x)$. The \emph{graph} of a
mapping $\Psi:X\sto Y$ is the set
$\left\{(x,y)\in X\times Y: y\in \Psi(x)\right\}$, and we say that
$\Psi$ has an \emph{open} (\emph{closed}) graph if its graph is open
(respectively, closed) in $X\times Y$. Finally, to each
$\varepsilon>0$ and a mapping $\Phi:X\sto Y$ into a metric space
$(Y,\rho)$, we will associate the mapping
$\mathbf{O}_\varepsilon[\Phi]:X\sto Y$ defined by
$\mathbf{O}_\varepsilon[\Phi](x)= \mathbf{O}_\varepsilon(\Phi(x))$,
$x\in X$.  This convention will be also used in an obvious manner for
usual maps $f:X\to Y$ considering $f$ as the singleton-valued mapping
$x\to\{f(x)\}$, $x\in X$.  In these terms, for  $f,g:X\to Y$ and
$\varepsilon,\mu>0$, we have that $f$ is a $\mu$-selection for
$\Phi:X\sto Y$ with $\rho(f(x),g(x))<\varepsilon$ for every $x\in X$, if
and only if $f$ is a selection for the mapping
$\mathbf{O}_\mu[\Phi]\wedge\mathbf{O}_\varepsilon[g]$. \medskip

The following three observations are due to Michael, see
\cite[Propositions 2.3 and 2.5]{michael:56a} and \cite[Lemma
11.3]{michael:56b}. For a paracompact domain, they reduce the
selection problem for l.s.c.\ mappings to that one of approximate
selections (see Proposition \ref{proposition-shsa-vgg-rev:1}).

\begin{proposition}[\cite{michael:56a}]
  \label{proposition-shsa-v28:2}
  For spaces $X$ and $Y$, a mapping $\Phi:X\sto Y$ is l.s.c.\ if and
  only if so is the mapping $\overline{\Phi}:X\to \mathscr{F}(Y)$.
\end{proposition}

\begin{proposition}[\cite{michael:56a}]
  \label{proposition-shsa-v28:5}
  Let $\Phi:X\sto Y$ be an l.s.c.\ mapping and $\Psi:X\sto Y$ be an
  open-graph mapping with $\Phi(x)\cap \Psi(x)\neq \emptyset$, for all
  $x\in X$. Then the mapping $\Phi\wedge\Psi:X\sto Y$ is also l.s.c.
\end{proposition}

\begin{proposition}[\cite{michael:56b}]
  \label{proposition-shsa-v28:4}
  If $(Y,\rho)$ is a metric space, $\varepsilon>0$ and $\Phi:X\sto Y$
  is l.s.c., then the associated mapping
  $\mathbf{O}_\varepsilon[\Phi]:X\sto Y$ has an open graph. In
  particular, for each compact set $K\subset Y$, the set
  $\left\{x\in X: K\subset \mathbf{O}_\varepsilon[\Phi](x)\right\}$ is
  open in $X$.
\end{proposition}

We now have the following general observation which is a partial case
of \cite[Proof that Lemma 5.1 implies Theorem 4.1,
page 569]{michael:56b}.

\begin{proposition}[\cite{michael:56b}]
  \label{proposition-shsa-vgg-rev:1}
    Let $E$ be a Banach space. Then for a space $X$, the following are
  equivalent\textup{:}
  \begin{enumerate}
  \item\label{item:shsa-vgg-rev:2} Each l.s.c.\ mapping
    $\Phi:X\to \mathscr{F}_\mathbf{c}(E)$ has a continuous selection.
  \item\label{item:shsa-vgg-rev:3} Each l.s.c.\ convex-valued mapping
    $\Phi:X\sto E$ has a continuous $\varepsilon$-selection, for every
    $\varepsilon>0$.
  \end{enumerate}
\end{proposition}

\begin{proof}
  If $\Phi:X\sto E$ is an l.s.c.\ convex-valued mapping, then by
  Proposition \ref{proposition-shsa-v28:2}, so is
  $\overline{\Phi}:X\to \mathscr{F}_\mathbf{c}(E)$. Moreover, if
  $f:X\to E$ is a selection for $\overline{\Phi}$, then $f$ is an
  $\varepsilon$-selection for $\Phi$, for every $\varepsilon>0$.  This
  shows that
  \ref{item:shsa-vgg-rev:2}$\ \Rightarrow\
  $\ref{item:shsa-vgg-rev:3}. To see that
  \ref{item:shsa-vgg-rev:3}$\ \Rightarrow\ $\ref{item:shsa-vgg-rev:2},
  suppose that \ref{item:9} holds, and
  $\Phi:X\to \mathscr{F}_\mathbf{c}(E)$ is l.s.c. Then by
  \ref{item:shsa-vgg-rev:3}, $\Phi$ has a continuous
  $2^{-1}$-selection $f_1:X\to E$. According to Propositions
  \ref{proposition-shsa-v28:5} and \ref{proposition-shsa-v28:4},
  $\Phi\wedge \mathbf{O}_{2^{-1}}[f_1]:X\sto E$ is also l.s.c., and
  clearly it is convex-valued. Hence, for the same reason, it has a
  continuous $2^{-2}$-selection $f_2:X\to E$.  Evidently, $f_2$ is a
  continuous $2^{-2}$-selection for $\Phi$ with
  $\|f_2(x)-f_1(x)\|<2^{-2}+2^{-1}<2^0$, $x\in X$. The construction
  can be carried on by induction to get a sequence $f_n:X\to E$,
  $n\in\N$, of continuous $2^{-n}$-selections for $\Phi$ such that
  $\|f_{n+1}(x)-f_n(x)\|< 2^{-n+1}$, $n\in \N$. This resulting
  sequence of maps is uniformly Cauchy, hence it converges to some
  continuous map $f:X\to E$ because $E$ is a Banach space; and this
  $f$ is a continuous selection for $\Phi$ as required in
  \ref{item:shsa-vgg-rev:2}.
\end{proof}

Complementary to Proposition \ref{proposition-shsa-vgg-rev:1} is the
following consequence of Theorems \ref{theorem-res-pro-v1:2} and
\ref{theorem-res-pro-v1:4}.
\begin{corollary}
  \label{corollary-res-pro-v11:1}
  Let $E$ be a Banach space, $X$ be a collectionwise normal space and
  $\Phi:X\sto E$ be an l.s.c.\ convex-valued mapping such that
  $\Phi(x)$ is compact, for every $x\in X$ with
  $\mathbf{0}\notin \Phi(x)$. Then $\Phi$ has a continuous
  $\varepsilon$-selection, for every $\varepsilon>0$.
\end{corollary}

This brings the following alternative question, which represents
another aspect of Question \ref{question-shsa-v15:1} (hence, of
Question \ref{question-res-pro-v1:1} as well).

\begin{question}
  \label{question-shsa-v14:1}
  Let $X$ be a collectionwise normal space, $E$ be a Banach space and
  $\Phi:X\to \mathscr{F}_\mathbf{c}(E)$ be an l.s.c.\ mapping which
  has a continuous $\varepsilon$-selection, for every
  $\varepsilon>0$. Does $\Phi$ have a continuous selection?
\end{question}

Evidently, the answer is ``Yes'' if $\Phi$ is a
$\mathscr{C}'_\mathbf{c}(E)$-valued mapping. Moreover, the property in
Question \ref{question-shsa-v14:1} implies collectionwise normality.

\begin{proposition}
  Let $X$ be a space such that for every Banach space $E$, every
  l.s.c.\ mapping $\Phi:X\to \mathscr{C}'_\mathbf{c}(E)$ has a
  continuous $\varepsilon$-selection, for every $\varepsilon>0$. Then
  $X$ is collectionwise normal.
\end{proposition}

\begin{proof}
  Let $A\subset X$ be closed, $E$ be a Banach space and $g:A\to E$ be
  a continuous map. Then the mapping
  $\Phi_g:X\to \mathscr{C}'_\mathbf{c}(E)$, defined as in Proposition
  \ref{proposition-contr-ext-14:1}, is l.s.c. Hence, by condition,
  $\Phi_g$ has a continuous $\varepsilon$-selection, for every
  $\varepsilon>0$. In other words, for every $\varepsilon>0$, the
  mapping $\Phi(x)=E$, $x\in X$, has a continuous selection $h:X\to E$
  such that $\|h(x)-g(x)\|<\varepsilon$, for all $x\in A$. According
  to \cite[Lemma 4.2]{gutev-ohta-yamazaki:06}, $g$ can be extended to
  a continuous map $f:X\to E$. Hence, by Theorem
  \ref{theorem-res-pro-v1:1}, $X$ is collectionwise normal.
\end{proof}

If $X$ is only assumed to be PF-normal, then the answer to Question
\ref{question-shsa-v14:1} is ``Yes'' if each $\Phi(x)$ is either
compact or finite-dimensional. The latter means that $\Phi(x)$ is
contained in some finite-dimensional affine subspace of $E$;
equivalently, that ${\Phi(x)\subset q+L}$ for some $q\in E$ and a
finite-dimensional linear subspace $L\subset E$.

\begin{proposition}
  \label{proposition-res-pro-v11:3}
  Let $X$ be a PF-normal space, $E$ be a normed space and
  ${\Phi:X\to \mathscr{F}_\mathbf{c}(E)}$ be an l.s.c.\ mapping
  such that each $\Phi(x)$ is either compact or finite-dimensional. If
  $\Phi$ has a continuous $\varepsilon$-selection for some
  $\varepsilon>0$, then it also has a continuous selection.
\end{proposition}

\begin{proof}
  Let $g:X\to E$ be a continuous $\varepsilon$-selection for $\Phi$,
  for some $\varepsilon>0$. Then by by Propositions
  \ref{proposition-shsa-v28:5} and \ref{proposition-shsa-v28:4}, the
  intersection mapping $\Phi\wedge\mathbf{O}_\varepsilon[g]:X\sto E$
  is also l.s.c.  Hence, by Proposition \ref{proposition-shsa-v28:2},
  so is the mapping
  $\varphi=\overline{\Phi\wedge\mathbf{O}_\varepsilon[g]}:X\to
  \mathscr{F}_\mathbf{c}(E)$.  Furthermore, $\varphi$ is
  compact-valued. Indeed, if $\Phi(x)$ is compact, then so is
  $\varphi(x)$. If $\Phi(x)$ is finite-dimensional, then
  $\Phi(x)\subset q+L$ for some finite-dimensional linear subspace
  $L\subset E$ and $q\in E$. However, $L$ is complete with respect to
  the norm of $E$ (being finite-dimensional) and
  $\varphi(x)-q\subset L$ is closed and bounded in $L$. Therefore,
  $\varphi(x)$ is compact being a translate of the compact set
  $\varphi(x)-q$. Thus, by Theorem \ref{theorem-res-pro-v1:3},
  $\varphi$ has a continuous selection $f:X\to E$. This $f$ is also a
  selection for $\Phi$ because $\varphi(x)\subset \Phi(x)$, for every
  $x\in X$.
\end{proof}

For an infinite cardinal number $\tau$, a space $X$ is called
\emph{$\tau$-paracompact} if every open cover of $X$ of cardinality
$\leq\tau$ has a locally finite open refinement. The
$\omega$-paracompact spaces are commonly called \emph{countably
  paracompact}.  A mapping $\varphi:X\sto E$ is a \emph{set-valued
  selection} (or \emph{set-selection}, or \emph{multi-selection}) for
a mapping $\Phi:X\sto E$ if $\varphi(x)\subset \Phi(x)$, for all
$x\in X$. A mapping $\varphi:X\sto E$, into a metric space $(E,d)$, is
\emph{bounded} if each $\varphi(x)$, $x\in X$, is a bounded subset of
$(E,d)$. In these terms, the property in Proposition
\ref{proposition-res-pro-v11:3} that ``$\Phi$ has a continuous
$\varepsilon$-selection for some $\varepsilon>0$'' was used to
construct a bounded-valued l.s.c.\ selection
$\varphi:X\to \mathscr{F}_\mathbf{c}(E)$ for $\Phi$.  Regarding this,
let us mention the following property of countable paracompactness,
see \cite[Lemma 2.1]{MR1779519} and \cite[Lemma 1.2]{MR2406397}.

\begin{proposition}
  \label{proposition-shsa-v12:1}
  Let $E$ be a normed space and $X$ be a countably paracompact
  space. Then every l.s.c.\ mapping
  $\Phi : X\to \mathscr{F}_\mathbf{c}(E)$ has a continuous selection
  if and only if every bounded l.s.c.\ mapping
  $\varphi : X\to \mathscr{F}_\mathbf{c}(E)$ has a continuous
  selection.
\end{proposition}

This brings the following characterisation of countably paracompact
PF-normal spaces, compare with Proposition
\ref{proposition-res-pro-v11:3}.

\begin{theorem}
  \label{theorem-shsa-v12:2}
  A space $X$ is countably paracompact and PF-normal if and only if
  for every normed space $E$, every l.s.c.\ mapping
  ${\Phi:X\to \mathscr{F}_\mathbf{c}(E)}$ with $\Phi(x)$ being either
  compact or finite-dimensional, has continuous selection.
\end{theorem}

\begin{proof}
  In the one direction, this is Theorem \ref{theorem-res-pro-v1:3} and
  Proposition \ref{proposition-shsa-v12:1} because each closed bounded
  subset of a finite-dimensional normed space is compact, see the
  proof of Proposition \ref{proposition-res-pro-v11:3}.  Conversely,
  if $X$ has the selection property in the theorem, then it is
  PF-normal (by Theorem \ref{theorem-res-pro-v1:3}). By taking $E$ to
  be the real line $\R$ and using \cite[Theorem
  3.1$^{\prime\prime}$]{michael:56a}, $X$ is also countably
  paracompact.
\end{proof}

A function $\xi:X\to \R$ is \emph{lower} (\emph{upper})
\emph{semi-continuous} if the set
\[
  \{x\in X:\xi(x)>r\}\quad \text{(respectively,
    $\{x\in X:\xi(x)<r\}$)}
\]
is open in $X$, for every $r\in \R$. If $(E,d)$ is a metric space,
$\varphi:X\sto E$ and $\eta:X\to (0,+\infty)$, then we shall say that
$g:X\to E$ is an \emph{$\eta$-selection} for $\varphi$ if
$g(x)\in \mathbf{O}_{\eta(x)}(\varphi(x))$, for every
$x\in X$. According to Dowker-Kat\v{e}tov's insertion theorem
\cite{dowker:51,katetov:51}, see also \cite[5.5.20(a)]{engelking:89},
a space $X$ is normal and countably paracompact if and only if for
every pair $\xi,\eta:X\to \R$ of functions such that $\xi$ is upper
semi-continuous, $\eta$ is lower semi-continuous and $\xi<\eta$, there
exists a continuous function $f:X\to \R$ with $\xi< f< \eta$. The
following selection interpretation of this insertion property was
shown in \cite[Theorem 4.3]{MR2643824}.

\begin{theorem}[\cite{MR2643824}]
  \label{theorem-shsa-v14:1}
  A space $X$ is countably paracompact and collectionwise normal if
  and only if for every Banach space $E$, l.s.c.\ mapping
  $\Phi:X\to \mathscr{C}'_\mathbf{c}(E)$, lower semi-continuous function
  $\eta:X\to (0,+\infty)$ and continuous $\eta$-selection ${g:X\to E}$
  for $\Phi$, there exists a continuous selection $f:X\to E$ for $\Phi$
  with $\|f(x)- g(x)\|<\eta(x)$, for all $x\in X$.
\end{theorem}

Without the assumption of countable paracompactness, the following
similar characterisation holds, see \cite[Proposition 4.2]{MR2643824}.

\begin{proposition}[\cite{MR2643824}]
  \label{proposition-shsa-v18:1}
  Let $X$ be a collectionwise normal space, $E$ be a Banach space,
  $\Phi:X\to \mathscr{C}'_\mathbf{c}(E)$ be l.s.c.,
  $\eta:X\to (0,+\infty)$ be continuous and $g:X\to E$ be a
  continuous $\eta$-selection for $\Phi$. Then $\Phi$ has a continuous
  selection $f:X\to E$ with $\|f(x)-g(x)\|\leq \eta(x)$, for all
  $x\in X$.
\end{proposition}

Motivated by this, the following question was posed in \cite[Question
3]{MR2643824}.

\begin{question}[\cite{MR2643824}]
  \label{question-shsa-v14:2}
  Let $X$ be a collectionwise normal space, $E$ be a Banach space,
  ${\Phi:X\to \mathscr{C}'_\mathbf{c}(E)}$ be an l.s.c.\ mapping and
  $g:X\to E$ be a continuous $\eta$-selection for $\Phi$, for some
  lower semi-continuous function ${\eta:X\to (0,+\infty)}$. Does
  $\Phi$ have a continuous selection $f:X\to E$ with
  $\|f(x)-g(x)\|\leq \eta(x)$, for all $x\in X$?
\end{question}

\begin{remark}
  \label{remark-shsa-v18:1}
  Let us point out that the answer to Question
  \ref{question-shsa-v14:2} is ``Yes'' if so is the answer to Question
  \ref{question-res-pro-v1:1}. Indeed, let $\Phi$, $\eta$ and $g$ be
  as in Question \ref{question-shsa-v14:2}. Then the mapping
  $\mathbf{O}_\eta[g]:X\sto E$, defined by
  $\mathbf{O}_\eta[g](x)=\mathbf{O}_{\eta(x)}(g(x))$,
  $x\in X$, has an open graph \cite[Proposition 2.1]{gutev:05}, see
  also Proposition \ref{proposition-shsa-v28:4}. Hence, the mapping
  $\varphi=\overline{\Phi\wedge \mathbf{O}_\eta[g]}$ remains
  convex-valued and l.s.c., by Propositions
  \ref{proposition-shsa-v28:2} and
  \ref{proposition-shsa-v28:5}. Moreover, $g(x)\notin \varphi(x)$
  implies that $\Phi(x)\neq E$ and, therefore, $\varphi(x)$ is compact
  being a closed subset of the compact set $\Phi(x)$. Evidently, if
  $f:X\to E$ is a continuous selection for $\varphi$, then
  $\|f(x)-g(x)\|\leq \eta(x)$ for all $x\in X$.\hfill\textsquare
\end{remark}

The property stated in Question \ref{question-shsa-v14:2} can be
considered as a selection interpretation of the classical
Kat\v{e}tov-Tong insertion theorem
\cite{katetov:51,MR0060211,tong:48,MR0050265}, see also
\cite[1.7.15(b)]{engelking:89}, that a space $X$ is normal if and only
if for every pair $\xi,\eta:X\to \R$ of functions such that $\xi$ is
upper semi-continuous, $\eta$ is lower semi-continuous and
$\xi\leq\eta$, there exists a continuous function $f:X\to \R$ with
$\xi\leq f\leq \eta$.

\section{Selections and Continuity-Like Properties}

\subsection{Selection Factorisation Properties}

Here, we briefly discuss two para\-compact-like properties of
set-valued mappings which offer a natural generalisation of Theorems
\ref{theorem-res-pro-v1:2} and \ref{theorem-res-pro-v9:1}. They are
based on the following idea of {\em factorising\/} set-valued
mappings. For a metrizable space $Y$ and a mapping $\Phi:X\sto Y$, we
say that a triple $(Z,h,\varphi)$ is an \emph{l.s.c.\
  weak-factorisation for} $\Phi$ \cite{choban-nedev:74,nedev:80} if
$Z$ is a metrizable space with weight $w(Z)\leq w(Y)$, $h:X\to Z$ is
continuous and $\varphi:Z\sto Y$ is l.s.c.\ such that
$\varphi\circ h:X\sto Y$ is a set-valued selection for
$\Phi$. \medskip

In what follows, using the convex hull operator $A\to \conv(A)$ of a
linear space $E$, to each mapping $\Phi:X\sto E$ we will associate the
mapping $\conv[\Phi]:X\sto E$, defined by
$\conv[\Phi](x)=\conv(\Phi(x))$, $x\in X$. It is well known that lower
semi-continuity is preserved by passing to the mapping $\conv[\Phi]$,
\cite[Proposition 2.6]{michael:56a}.

\begin{proposition}[\cite{michael:56a}]
  \label{proposition-shsa-v28:3}
  Let $E$ be a normed space and $\Phi:X\sto E$ be an l.s.c.\
  mapping. Then the mapping $\conv[\Phi]:X\sto E$ is also l.s.c. 
\end{proposition}

We now have the following general reduction of the selection problem
for set-valued mappings.

\begin{proposition}
  \label{proposition-shsa-v17:1}
  Let $X$ be a space, $E$ be a Banach space and
  $\Phi:X\to \mathscr{F}_\mathbf{c}(E)$ be a mapping which admits an
  l.s.c.\ weak-factorisation. Then $\Phi$ has a continuous selection.
\end{proposition}

\begin{proof}
  Let $(Z,h,\varphi)$ be an l.s.c.\ weak-factorisation for $\Phi$. By
  Propositions \ref{proposition-shsa-v28:2} and
  \ref{proposition-shsa-v28:3}, the associated mapping
  $\overline{\conv[\varphi]}:Z\to \mathscr{F}_\mathbf{c}(E)$ remains
  l.s.c.  Hence, by Theorem \ref{theorem-res-pro-v9:1}, it has a
  continuous selection $g:Z\to E$. Then $g\circ h:X\to E$ is a
  continuous selection for $\Phi$ because
  ${\overline{\conv[\varphi]}\circ h:X\to \mathscr{F}_\mathbf{c}(E)}$
  is a set-valued selection for $\Phi$.
\end{proof}

Considering l.s.c.\ weak-factorisations in the setting of Theorems
\ref{theorem-res-pro-v1:2} and \ref{theorem-res-pro-v9:1}, Nedev
\cite{nedev:80} (see also \cite{choban-nedev:74}) defined the
following property. A mapping $\Phi:X\sto Y$ is said to have the
\emph{Selection Factorisation Property} (called \emph{s.f.p.}, for
short) if for every closed subset $F\subset X$ and every locally
finite collection $\mathscr{U}$ of open subsets of $Y$ such that
$\Phi^{-1}[\mathscr{U}]= \left\{\Phi^{-1}[U]:U\in\mathscr{U}\right\}$
covers $F$, there exists a locally finite and open (in $F$) cover of
$F$ which refines $\Phi^{-1}[\mathscr{U}]$. The importance of s.f.p.\
mappings is evident from the following two observations.

\begin{example}[\cite{nedev:80}]
  \label{example-res-pro-v8:1}
  Let $\Phi:X\to \mathscr{F}(Y)$ be an l.s.c.\ mapping, where $Y$ is a
  metrizable space. Then $\Phi$ has the s.f.p.\ provided $X$ is
  $w(Y)$-paracompact, or when $X$ is $w(Y)$-collectionwise normal and
  $\Phi:X\to \mathscr{C}'(Y)$.
\end{example}

\begin{theorem}[\cite{choban-nedev:74,nedev:80}]
  \label{theorem-res-pro-v8:1}
  Let $X$ be a normal space and $Y$ be a completely metrizable
  space. Then each s.f.p.\ mapping $\Phi:X\to \mathscr{F}(Y)$ has an
  l.s.c.\ weak-factori\-sa\-tion.
\end{theorem}

To extend Theorem \ref{theorem-res-pro-v8:1} to set-valued mappings
defined on arbitrary spaces, the following similar property was
defined in \cite{gutev:00a}. Let $(Y,\rho)$ be a metric space.  A
mapping $\Phi:X\sto Y$ is said to be \emph{lower semi-factorisable
  relatively $\rho$}, or \emph{$\rho$-l.s.f.}, if for every closed
subset $F\subset X$, every $\varepsilon >0$ and every (not necessarily
continuous) selection $s:F\to Y$ for $\Phi \uhr F$, there exists a
locally finite (in $F$) covering $\mathscr{U}$ of $F$ of cozero-sets
of $F$, and a map $\kappa:\mathscr{U}\to F$ such that
\[
  |\mathscr{U}|\leq w(Y)\quad \text{and}\quad s(\kappa(U))\in
  \mathbf{O}_\varepsilon\left(\Phi(x)\right),\ \text{for every
    $x\in U\in \mathscr{U}$.}
\]

Here are two important properties of $\rho$-l.s.f.\ mappings.

\begin{example}[\cite{gutev:00a}]
  \label{example-res-pro-v8:2}
  Let $X$ be a normal space, $(Y,\rho)$ be a metric space and
  $\Phi:X\sto Y$ be an s.f.p.\ mapping. Then $\Phi$ is $\rho$-l.s.f.
\end{example}

\begin{theorem}[\cite{gutev:00a}]
  \label{theorem-res-pro-v8:2}
  Let $X$ be a space, $(Y,\rho)$ be a complete metric space and
  $\Phi:X\to \mathscr{F}(Y)$ be a $\rho$-l.s.f.\ mapping. Then $\Phi$
  has an l.s.c.\ weak-factorisation. 
\end{theorem}

\subsection{Continuity of Set-Valued Mappings}

Based on Theorem \ref{theorem-res-pro-v8:2}, the $\rho$-l.s.f.\
mappings deal with several other selection theorems which are similar
to Theorems \ref{theorem-res-pro-v1:2} and \ref{theorem-res-pro-v9:1},
but any restriction on the domain is removed at the expense of
strengthening the continuity of the set-valued mappings. This is
discussed below.\medskip

A mapping $\Phi:X\sto Y$ is called \emph{continuous} if it is both
l.s.c.\ and u.s.c. Here, $\Phi$ is \emph{u.s.c.}, or \emph{upper
  semi-continuous}, if the set if $\Phi^{-1}[F]$ is closed in $X$ for
every closed $F\subset Y$; equivalently, if
$\Phi^\#[U]=\big\{x\in X: \Phi(x)\subset U\big\}$ is open in $X$, for
every open $U\subset Y$.\medskip

Let $(Y,\rho)$ be a metric space. A mapping $\Phi:X\sto Y$ is
\emph{$\rho$-l.s.c.}  (\emph{$\rho$-u.s.c.})  if for every
$\varepsilon >0$, each point $p\in X$ has a neighbourhood $V$ such
that $\Phi(p)\subset \mathbf{O}_\varepsilon( \Phi ( x))$
(respectively, $\Phi(x) \subset \mathbf{O}_\varepsilon( \Phi (p))$),
for every $x\in V$.  Metric semi-continuity offers other
interpretations of continuity.  Namely, a mapping $\Phi:X\sto Y$ is
called
\begin{enumerate}[label=(\roman*)]
\item \emph{$\rho$-continuous} if it is both $\rho$-l.s.c.\ and
  $\rho$-u.s.c.;
\item \emph{$\rho$-proximal continuous} if it is both l.s.c.\ and
  $\rho$-u.s.c.;
\item \emph{proximal continuous} if it is $d$-proximal continuous for
  some metric $d$ on $Y$, which is topologically equivalent to $\rho$.
\end{enumerate}

In the realm of set-valued mappings with a metrizable range, every
continuous or $\rho$-continuous mapping is $\rho$-proximal continuous,
but the converse is not true, see \cite[Proposition
2.5]{gutev:00a}. Regarding continuity and $\rho$-continuity, it is
well known that these properties coincide precisely when the range is
a compact metric space, see e.g.\ \cite[Proposition
2.6]{gutev:00a}. \medskip

The continuous and $\rho$-continuous mappings fit naturally into the
selection theory. Selection theorems for $\rho$-continuous mappings
with paracompact (or even arbitrary) domain were obtained in
\cite{choban:70b,michael:57}, while selection results for continuous
mappings with arbitrary or (collectionwise) normal domain were
obtained in
\cite{choban:70a,engelking-heath-michael:68}. Subsequently, these
results were extended to proximal continuous mappings based on the
following example, see \cite[Example
4.3]{gutev:00a}.

\begin{example}[\cite{gutev:00a}]
  \label{example-shsa-vgg:1}
  If $X$ is a space and $(Y,\rho)$ is a metric space, then each
  proximal continuous mapping $\Phi:X\sto Y$ is $\rho$-l.s.f.
\end{example}

A subset $A\subset X$ is \emph{$P^\lambda$-embedded} in a space $X$,
where $\lambda$ is an infinite cardinal number, if for every locally
finite cozero-set cover $\mathscr{W}$ of $A$ of cardinality
$\left| \mathscr{W}\right| \leq \lambda $, there exists a locally
finite cozero-set cover $\mathscr{U}$ of $X$ such that $\mathscr{W} $
is refined by $\mathscr{U}\uhr A=\{ U\cap A:U\in \mathscr{U}\}$. The
notion ``$P^\lambda $-embedded'' in this sense is the same as
``$P^\lambda $-embedded'' in the sense of Shapiro \cite{shapiro:66},
which was introduced by Arens \cite{arens:53} under the name
``$\lambda$-normally embedded'' (see \cite{shapiro:66}). It is well
known that every continuous map from a $P^\lambda$-embedded subset $A$
of a space $X$ into a Banach space $E$ of weight $w(E)\leq\lambda$, is
continuously extendable to the whole of $X$ (Al\'o and Sennott
\cite{alo-sennott:71}, Morita \cite{morita:75}, Przymusi\'nski
\cite{przymusinski:78}). However, in the setting of arbitrary spaces,
this extension property cannot be covered by the framework of
continuous set-valued mapping. Indeed, for a Banach space $E$, a
$P^{w(E)}$-embedded set $A\subset X$ and a continuous map $g:A\to E$,
the mapping $\Phi_g:X\to \mathscr{F}(E)$ defined as in Proposition
\ref{proposition-contr-ext-14:1} may fail to be even l.s.c.\medskip

To rectify this, the following approach was offered in
\cite{gutev:00a}.  A map $g:A\to Y$, where $A\subset X$, is called
\emph{$A$-regular} \cite{gutev:00a} if for every locally finite
cozero-set cover $\mathscr{V}$ of $Y$, there exists a locally finite
cozero-set cover $\mathscr{U}$ of $X$ such that $g(\mathscr{U})$
refines $\mathscr{V}$. A continuous $g : A\to Y$ is $A$-regular,
whenever $A$ is $P^{w(Y)}$-embedded in $X$. Moreover, the restriction
$f\uhr A$ is $A$-regular, for every continuous map $f:X\to Y$ into a
metrizable space $Y$. The following improved version of Example
\ref{example-shsa-vgg:1} was obtained in \cite[Example
4.4]{gutev:00a}.

\begin{example}[\cite{gutev:00a}]
  \label{example-shsa-v17:1}
  Let $X$ be a space, $(Y,\rho)$ be a metric space,
  $\Phi :X\to \mathscr{F}(Y)$ be proximal continuous, $A\subset X$ and
  $g:A\rightarrow Y$ be an $A$-regular selection for $\Phi\uhr
  A$. Define $\Phi_g:X\to \mathscr{F}(Y)$ by $\Phi_g(x)=\{g(x)\}$ if
  $x\in A$, and $\Phi_g(x)=\Phi(x)$ otherwise.  Then the mapping
  $\Phi_g$ is $\rho $-l.s.f.
\end{example}

Combining this example with Proposition \ref{proposition-shsa-v17:1}
and Theorem \ref{theorem-res-pro-v8:2}, the following general result
was obtained in \cite[Corollary 6.2]{gutev:00a}.

\begin{theorem}[\cite{gutev:00a}]
  \label{theorem-res-pro-v1:5}
  Let $X$ be a space, $E$ be a Banach space,
  $\Phi : X \to \mathscr{F}_\mathbf{c}(E)$ be proximal continuous,
  $A\subset X$ and $g : A\to E$ be a continuous selection for
  $\Phi\uhr A$. Then $g$ can be extended to a continuous selection for
  $\Phi$ if and only if $g$ can be extended to a continuous map
  $f : X \to E$.
\end{theorem}

\subsection{Hybrid Continuity and Selections}

The following result was obtained in \cite[Lemma
4.2]{gutev-ohta-yamazaki:03}. 

\begin{theorem}[\cite{gutev-ohta-yamazaki:03}]
  \label{theorem-shsa-v18:2}
  Let $X$ be a collectionwise normal space, $E$ be a Banach space,
  $\Phi : X\to \mathscr{F}(E)$ be proximal continuous and
  $\varphi: X\to \mathscr{F}_\mathbf{c}(E)$ be an l.s.c.\ selection
  for $\Phi$ such that $\varphi(x)$ is compact, whenever
  $\varphi(x)\neq \Phi(x)$. Then $\varphi$ has a continuous selection.
\end{theorem}

Let us explicitly remark that the proof of Theorem
\ref{theorem-shsa-v18:2} is based on the fact that, in this case, the
mapping $\varphi:X\to \mathscr{F}_\mathbf{c}(E)$ has the s.f.p. This
idea is extended in the following similar result.

\begin{theorem}
  \label{theorem-shsa-v18:1}
  Let $X$ be a collectionwise normal space, $E$ be a Banach space,
  $\Phi : X\to \mathscr{F}_\mathbf{c}(E)$ be l.s.c.\ and
  $A\subset X$ be a closed subset such that $\Phi\uhr X\setminus A$
  is proximal continuous. Then each continuous selection $g:A\to E$
  for $\Phi\uhr A$ can be extended to a continuous selection for
  $\Phi$.
\end{theorem}

\begin{proof}
  Let $g:A\to E$ be a continuous selection for $\Phi\uhr A$. Define
  ${\Phi_g:X\to \mathscr{F}_\mathbf{c}(E)}$ as in Example
  \ref{example-shsa-v17:1}, namely $\Phi_g(x)=\{g(x)\}$ if $x\in A$
  and $\Phi_g(x)=\Phi(x)$ otherwise. The proof consists of showing
  that $\Phi_g$ has the s.f.p. To this end, take a closed set
  $F\subset X$, and a locally finite family $\mathscr{U}$ of open
  subsets of $E$ with ${F\subset
    \bigcup\Phi_g^{-1}[\mathscr{U}]}$. Since $X$ is collectionwise
  normal, by Theorem \ref{theorem-res-pro-v1:1}, $g$ can be extended
  to a continuous map $h:X\to E$. Then the family
  $\mathscr{V}_0=\{\Phi^{-1}_g[U]\cap h^{-1}(U): U\in \mathscr{U}\}$
  is open and locally finite in $X$, and refines
  $\Phi_g^{-1}[\mathscr{U}]$. Moreover,
  $F_0=F\cap A\subset \bigcup\mathscr{V}_0$ because $h\uhr A=g$. Set
  $F_1=F\setminus \bigcup\mathscr{V}_0$ and
  $\varphi=\Phi_g\uhr X\setminus A$, so that
  $F_1\subset \bigcup\varphi^{-1}[\mathscr{U}]$. Since
  $\varphi=\Phi\uhr X\setminus A$ is proximal continuous, by
  \cite[Theorem 3.1]{gutev:00a}, $\varphi^{-1}[\mathscr{U}]$ is
  refined by a $\sigma$-discrete (in $X\setminus A$) family
  $\mathscr{W}$ of cozero-sets of $X\setminus A$ such that
  $\bigcup\mathscr{W}=\bigcup\varphi^{-1}[\mathscr{U}]$. Since $X$ is
  normal, there exists an open set $G\subset X$ with
  $F_1\subset G\subset \overline{G}\subset
  \bigcup\varphi^{-1}[\mathscr{U}]=\bigcup\mathscr{W}$. Finally, using
  \cite[Theorem 1.2]{morita:60} and \cite[Theorem 1.2]{morita:64},
  take an open and locally finite (in $X\setminus A$) cover
  $\mathscr{W}_1$ of $\overline{G}$, which refines $\mathscr{W}$. Then
  the family $\mathscr{V}_1=\{W\cap G: W\in \mathscr{W}_1\}$ is open
  and locally finite in $X$. Moreover, it refines
  $\Phi_g^{-1}[\mathscr{U}]$ and covers $F_1$ because
  \[
    F_1\subset G=\bigcup \mathscr{V}_1\subset \bigcup
    \mathscr{W}=\bigcup\varphi^{-1}[\mathscr{U}]\subset
    \bigcup\Phi_g^{-1}[\mathscr{U}].
  \]
  Accordingly, $\mathscr{V}=\mathscr{V}_0\cup \mathscr{V}_1$ is a
  locally finite family of open subset of $X$, which refines
  $\Phi_g^{-1}[\mathscr{U}]$ and covers $F$. Hence, $\Phi_g$ has the
  s.f.p.\ and by Proposition \ref{proposition-shsa-v17:1} and Theorem
  \ref{theorem-res-pro-v8:1}, it also has a continuous
  selection. Thus, $g$ can be extended to a continuous selection for
  $\Phi$.
\end{proof}

Theorems \ref{theorem-shsa-v18:2} and \ref{theorem-shsa-v18:1} are
naturally related to $\mathscr{C}'(E)$-valued mappings. Turning to
this, to each mapping ${\Phi:X\to \mathscr{F}(E)}$ we will associate
the sets
\begin{equation}
  \label{eq:res-pro-v7:1}
  \mathbf{C}_\Phi=\left\{x\in X: \Phi(x)\in
  \mathscr{C}(E)\right\}\quad\text{and}\quad
  \mathbf{D}_\Phi=X\setminus \mathbf{C}_\Phi.
\end{equation}

Here are some properties of these sets in the setting of 
$\mathscr{C}'(E)$-valued l.s.c.\ mappings.

\begin{proposition}
  \label{proposition-res-pro-v7:1}
  Let $E$ be a metrizable space and $\Phi:X\to \mathscr{C}'(E)$ be an
  l.s.c.\ mapping. Then
  \begin{enumerate}[label=\upshape{(\roman*)}]
  \item\label{item:5} $\Phi$ is continuous at each point
    $x\in \mathbf{D}_\Phi$\textup{;}
  \item\label{item:6} $\mathbf{C}_\Phi$ is an
    $F_\sigma$-set\textup{;}
  \item\label{item:7} If\/ $\Phi[\mathbf{C}_\Phi]=\bigcup_{x\in
      \mathbf{C}_\Phi}\Phi(x)$ is not dense in $E$, then
    $\mathbf{C}_\Phi$ is closed in $X$.
  \end{enumerate}
\end{proposition}

\begin{proof}
  If $p\in \mathbf{D}_\Phi$, then $\Phi^{\#}[\Phi(p)]=\Phi^{\#}[E]=X$
  is open. This shows \ref{item:5}. If $E$ is compact, then
  $\mathbf{D}_\Phi=\emptyset$. If not, take a countable locally finite
  open cover $\mathscr{V}$ of $E$, which has no finite subcover. Then
  $\mathbf{D}_\Phi=\bigcap\Phi^{-1}[\mathscr{V}]$, which shows
  \ref{item:6}. To show finally \ref{item:7}, suppose that
  $U\subset E$ is a nonempty open set with $U\cap \Phi(x)=\emptyset$,
  for every $x\in \mathbf{C}_\Phi$. Then
  $\mathbf{D}_\Phi= \Phi^{-1}[U]$ is open, hence $\mathbf{C}_\Phi$ is
  closed.
\end{proof}

Based on this proposition and Theorems \ref{theorem-shsa-v18:2} and
\ref{theorem-shsa-v18:1}, see also Question
\ref{question-res-pro-v1:1}, we have the following question relating
the set $\mathbf{C}_\Phi$ to the selection problem for
collectionwise normal spaces.

\begin{question}
  \label{question-res-pro-v1:2}
  Let $X$ be a collectionwise normal space, $E$ be a Banach space and
  $\Phi:X\to \mathscr{F}_\mathbf{c}(E)$ be an l.s.c.\ mapping, which
  is continuous at each point $p\in X$ with
  $\Phi(p)\notin \mathscr{C}(E)$. Does $\Phi$ have a continuous
  selection?
\end{question}

Let $X$ be a space, $E$ be a Banach space,
$\Phi:X\to \mathscr{F}_\mathbf{c}(E)$ be an l.s.c.\ mapping and
$\mathbf{D}_\Phi$ be as in \eqref{eq:res-pro-v7:1}. If
$\Phi\uhr \mathbf{D}_\Phi$ is continuous, then it has a continuous
selection $g:\mathbf{D}_\Phi\to E$ \cite[Theorem 6.1]{gutev:00a}, see
also Theorem \ref{theorem-res-pro-v1:5}. Since $E$ is a Banach space
(hence, a complete metric space), $g$ can be extended to a continuous
map on some $G_\delta$-subset of $X$ containing $\mathbf{D}_\Phi$, see
\cite[Theorem 4.3.21]{engelking:89}. Thus, according to Proposition
\ref{proposition-res-pro-v11:1}, we get the following consequence.

\begin{corollary}
  \label{corollary-res-pro-v3:1}
  Let $X$ be a space, $E$ be a Banach space and
  $\Phi:X\to \mathscr{F}_\mathbf{c}(E)$ be an l.s.c.\ mapping such
  that $\Phi\uhr \mathbf{D}_\Phi$ is continuous. Then $\Phi\uhr H$ has
  a continuous selection for some $G_\delta$-subset $H\subset X$ with
  $\mathbf{D}_\Phi\subset H$.
\end{corollary}

On the other hand, collectionwise normality is hereditary with respect
to $F_\sigma$-sets \cite[Theorem 1.3]{MR0112115}. So, complementary to
Corollary \ref{corollary-res-pro-v3:1} is the following consequence of
Theorems \ref{theorem-res-pro-v1:2} (see also Theorem
\ref{theorem-res-pro-v1:3}) for the set
$\mathbf{C}_\Phi=X\setminus \mathbf{D}_\Phi$.

\begin{corollary}
  \label{corollary-shsa-vgg-rev:1} Let $X$ be a collectionwise normal
  space, $E$ be a normed space and
  ${\Phi:X\to \mathscr{F}_\mathbf{c}(E)}$ be an l.s.c.\ mapping. Then
  $\Phi\uhr Z$ has a continuous selection, for every $F_\sigma$-set
  $Z\subset X$ with $Z\subset\mathbf{C}_\Phi$.
\end{corollary}

Corollaries \ref{corollary-res-pro-v3:1} and
\ref{corollary-shsa-vgg-rev:1} give a good illustration of Question
\ref{question-res-pro-v1:2} showing that one can construct two partial
continuous selections for $\Phi$ on complementary subsets of the
domain. Hence, the question is if one can use these partial
selections, or other information, to construct a continuous selection
for the mapping $\Phi$ itself. However, it is not so likely that any
one of these partial selections can be extended to a continuous
selection for $\Phi$.

\begin{proposition}
  \label{proposition-res-pro-v1:4}
  Let $\Phi:X\to \mathscr{F}(E)$ be an l.s.c.\ mapping, $g:X\to E$ be
  a continuous map and $A=\left\{x\in X: g(x)\in \Phi(x)\right\}$. If
  $g\uhr A$ can be extended to a continuous selection $f:X\to E$ for
  $\Phi$, then $A$ must be closed.
\end{proposition}

\begin{proof}
  Follows from the fact that, in this case, $A=\{x\in X: g(x)=f(x)\}$.
\end{proof}

\subsection{Densely Compact-Valued Mappings}

For a mapping $\Phi:X\sto E$ and a subset $A\subset X$, let
$\Phi[A]=\bigcup_{x\in A}\Phi(x)$ (see Proposition
\ref{proposition-res-pro-v7:1}). 

\begin{proposition}
  \label{proposition-res-pro-v4:1}
  Let $\Phi:X\sto E$ be an l.s.c.\ mapping and $A\subset X$. Then
  the set
  $Z=\left\{x\in X: \Phi(x)\subset \overline{\Phi[A]}\right\}$
  is closed in $X$. In particular, $\Phi[A]$ is dense $\Phi[X]$
  whenever $A$ is dense in $X$. 
\end{proposition}

\begin{proof}
  Simply observe that
  $X\setminus
  Z=\Phi^{-1}\left[E\setminus\overline{\Phi[A]}\right]$ is open.
\end{proof}

This brings the following refined version of Question
\ref{question-res-pro-v1:2}.

\begin{question}
  \label{question-res-pro-v1:3}
  Under the conditions of Question \ref{question-res-pro-v1:2},
  suppose further that each nonempty open set $U\subset X$ contains a
  point $p\in U$ with $\Phi(p)\in \mathscr{C}(E)$, i.e.\ that
  $\mathbf{C}_\Phi$ is dense in $X$, see \eqref{eq:res-pro-v7:1}. Does
  $\Phi$ have a continuous selection?
\end{question}

Let us remark that if the answer to Question
\ref{question-res-pro-v1:3} is ``Yes'', then by Theorem
\ref{theorem-shsa-v18:1}, so is the answer to Question
\ref{question-res-pro-v1:2}. Finally, here is a bit more general question. 

\begin{question}
  \label{question-res-pro-v1:4}
  Let $X$ be a collectionwise normal space, $E$ be a Banach space and
  $\Phi:X\to \mathscr{F}_\mathbf{c}(E)$ be an l.s.c.\ mapping such
  that each nonempty open set $U\subset X$ contains a point $p\in U$
  with $\Phi(p)\in \mathscr{C}(E)$. Does $\Phi$ have a continuous
  selection?
\end{question}

\subsection{Extending Selections}

The following simple observation, obtained in \cite[Example
1.3*]{michael:56a}, shows that the selection-extension problem for
l.s.c.\ mappings is equivalent to the selection problem for these
mappings.

\begin{proposition}[\cite{michael:56a}]
  \label{proposition-shsa-v28:1}
  Let $\Phi:X\sto Y$ be l.s.c., $A\subset X$ be closed and $g:A\to Y$
  be a continuous selection for $\Phi\uhr A$. Define $\Phi_g:X\sto Y$
  by $\Phi_g(x)=\{g(x)\}$ if $x\in A$, and $\Phi_g(x)=\Phi(x)$
  otherwise.  Then the mapping $\Phi_g$ also l.s.c.
\end{proposition}

This implies the following interpretation of Theorem
\ref{theorem-res-pro-v9:1}.

\begin{theorem}
  \label{theorem-res-pro-v5:1}
  Let $X$ be paracompact, $E$ be a Banach space,
  ${\Phi:X\to \mathscr{F}_\mathbf{c}(E)}$ be an l.s.c.\ mapping and
  $\Omega_\Phi$ be the collection of all continuous selections for
  $\Phi$. Then $\Phi(x)=\left\{f(x): f\in \Omega_\Phi\right\}$, for
  every $x\in X$.
\end{theorem}

Precisely the same interpretation holds for Theorem
\ref{theorem-res-pro-v1:2} as well.

\begin{theorem}
  \label{theorem-res-pro-v5:2}
  Let $X$ be a collectionwise normal space, $E$ be a Banach space,
  ${\Phi:X\to \mathscr{C}'_\mathbf{c}(E)}$ be an l.s.c.\ mapping and
  $\Omega_\Phi$ be the collection of all continuous selections for
  $\Phi$. Then $\Phi(x)=\left\{f(x): f\in \Omega_\Phi\right\}$, for
  every $x\in X$.
\end{theorem}

This brings the following natural question; see Theorems
\ref{theorem-res-pro-v1:5} and \ref{theorem-shsa-v18:1}, also
Proposition \ref{proposition-res-pro-v7:1}.

\begin{question}
  \label{question-res-pro-v5:2}
  Let $X$ be a collectionwise normal space, $E$ be a Banach space and
  ${\Phi:X\to \mathscr{F}_\mathbf{c}(E)}$ be an l.s.c.\ mapping which
  has a continuous selection $f:X\to E$. If $A\subset X$ is closed and
  $g:A\to E$ is a continuous selection for $\Phi\uhr A$, then is it
  possible to extend $g$ to a continuous selection for $\Phi$?
\end{question}

\section{Selections and Compact-Like Families}

\subsection{A General Selection Problem}
  
The selection problem for collectionwise normal spaces has two
aspects. The one is simply the question for a particular set-valued
mapping.

\begin{question}
  \label{question-res-pro-v1:8}
  Let $X$ be a collectionwise normal space, $E$ be a Banach space and
  $\Phi:X\to \mathscr{F}_\mathbf{c}(E)$ be an l.s.c.\ mapping. When
  does there exist a continuous selection for $\Phi$?
\end{question}

The other question is about a particular family
$\mathscr{L}(E)\subset \mathscr{F}_\mathbf{c}(E)$ with the property
that every l.s.c.\ mapping $\Phi:X\to \mathscr{L}(E)$ has a continuous
selection.

\begin{question}
  \label{question-res-pro-v1:9}
  Let $X$ be a collectionwise normal space and $E$ be a Banach
  space. Find a large enough subfamily
  $\mathscr{L}(E)\subset \mathscr{F}_\mathbf{c}(E)$ such that every
  l.s.c.\ mapping $\Phi:X\to \mathscr{L}(E)$ has a continuous
  selection?
\end{question}

Question \ref{question-res-pro-v1:9} is rather general, and to make
sense natural restrictions are in place. For instance, such a family
should include Dowker's extension theorem (Theorem
\ref{theorem-res-pro-v1:1}) in the sense of the construction in
Proposition \ref{proposition-contr-ext-14:1}. Therefore, one natural
requirement is that
\begin{equation}
  \label{eq:shsa-v18:1}
 \mathscr{C}_\mathbf{c}'(E)\subset \mathscr{L}(E). 
\end{equation}

Another natural condition may come from the construction of
approximate selections in Proposition
\ref{proposition-shsa-v18:1}. Namely, this proposition can be
rephrased in the following way.

\begin{proposition}
  \label{proposition-shsa-v18:2}
  Let $X$ be a collectionwise normal space, $E$ be a Banach space,
  $\Phi:X\to \mathscr{C}'_\mathbf{c}(E)$ be l.s.c.,
  $\eta:X\to (0,+\infty)$ be continuous and $g:X\to E$ be a
  continuous $\eta$-selection for $\varphi$. Then the mapping
  $\overline{\Phi\wedge \mathbf{O}_\eta[g]}:X\to
  \mathscr{F}_\mathbf{c}(E)$ has a continuous selection.
\end{proposition}

Here, for a mapping $\varphi:X\sto E$ and a function
$\eta:X\to (0,+\infty)$, the mapping
$\mathbf{O}_\eta[\varphi]:X\sto E$ is defined by
$\mathbf{O}_\eta[\varphi](x)=
\mathbf{O}_{\eta(x)}(\varphi(x))$, $x\in X$.  If $\varphi$
is l.s.c.\ (in particular, a usual continuous map) and $\eta$ is a
lower semi-continuous function, then $\mathbf{O}_\eta[\varphi]$ has an
open graph \cite[Proposition 2.1]{gutev:05}; see Remark
\ref{remark-shsa-v18:1} where this was already used. Thus, by
Propositions \ref{proposition-shsa-v28:2} and
\ref{proposition-shsa-v28:5}, the mapping
$\overline{\Phi\wedge \mathbf{O}_\eta[g]}$ in Proposition
\ref{proposition-shsa-v18:2} is also l.s.c. Hence, one can incorporate
the property by considering the following further condition on the
collection $\mathscr{L}(E)$.
\begin{equation}
  \label{eq:shsa-v18:2}
  \overline{S\cap \mathbf{O}_\delta(y)}\in \mathscr{L}(E),\ \text{whenever 
    $S\in \mathscr{L}(E)$ and $y\in \mathbf{O}_\delta(S)$ for some $\delta>0$.}
\end{equation}

Accordingly, we have the following refined version of Question
\ref{question-res-pro-v1:9}.

\begin{question}
  \label{question-res-pro-v1:10}
  Let $X$ be a collectionwise normal space, $E$ be a Banach space and
  $\mathscr{L}(E)\subset \mathscr{F}_\mathbf{c}(E)$ be as in
  (\ref{eq:shsa-v18:1}) and (\ref{eq:shsa-v18:2}). Then is it true
  that each l.s.c.\ mapping $\Phi:X\to \mathscr{L}(E)$ has a
  continuous selection?
\end{question}

Let us point out that Question \ref{question-res-pro-v1:10} is related
to Question \ref{question-shsa-v14:2}, see Remark
\ref{remark-shsa-v18:1}. Indeed, let
$\Phi:X\to \mathscr{C}'_\mathbf{c}(E)$, $\eta:X\to (0,+\infty)$ and
$g:X\to E$ be as in Question \ref{question-shsa-v14:2}. Then
$\overline{\Phi\wedge \mathbf{O}_\eta[g]}:X\to
\mathscr{F}_\mathbf{c}(E)$ is l.s.c.  Moreover, if
$\mathscr{L}(E)\subset \mathscr{F}_\mathbf{c}(E)$ satisfies
(\ref{eq:shsa-v18:1}) and (\ref{eq:shsa-v18:2}), then
$\overline{\Phi\wedge \mathbf{O}_\eta[g]}:X\to \mathscr{L}(E)$. Thus,
if the answer to Question \ref{question-res-pro-v1:10} is ``Yes'',
then $\overline{\Phi\wedge \mathbf{O}_\eta[g]}$ has a continuous
selection, and the answer to Question \ref{question-shsa-v14:2} will be
``Yes'' as well.

\subsection{A Necessary Condition}

Related to Question \ref{question-res-pro-v1:9}, the following
interesting result was obtained by Nedev and Valov
\cite{nedev-valov:84}.

\begin{theorem}[\cite{nedev-valov:84}]
  \label{theorem-res-pro-v2:1}
  Let $X$ be a normal space which is not countably paracompact, $E$ be
  a Banach space, and
  $\mathscr{L}(E)\subset \mathscr{F}_\mathbf{c}(E)$ be such that every
  l.s.c.\ mapping $\Phi:X\to \mathscr{L}(E)$ has a continuous
  selection. Then any decreasing sequence of elements of
  $\mathscr{L}(E)$ has a nonempty intersection.
\end{theorem}

\begin{proof}
  We present the proof in \cite[(b) of Theorem
  1]{nedev-valov:84}. Contrary to the claim, suppose that
  $\mathscr{L}(E)$ has a strictly decreasing sequence $\{F_n\}$ with
  an empty intersection. So, for every $n\in\N$, there is a point
  $z_n\in F_n\setminus F_{n+1}$. Then the set $H=\{z_n: n\in\N\}$ is
  closed and discrete in $F_1$ because
  $\bigcap_{n=1}^\infty F_n=\emptyset$. Next, for every $y\in F_1$,
  set $m(y)=\max\{n\in \N: y\in F_n\}$ and define an l.s.c.\ mapping
  $\varphi:F_1\to \mathscr{F}(H)$ by
  $\varphi(y)=\big\{z_k: k\geq m(y)\big\}$, $y\in F_1$. Since $F_1$ is
  paracompact (being metrizable), by a result of Michael
  \cite{MR0109343}, $\varphi$ has a u.s.c.\ selection
  $\psi:F_1\to \mathscr{C}(H)$. On the other hand, $X$ is normal but
  not countably paracompact. Hence, it has an increasing open cover
  $\{U_n\}$ which doesn't admit a closed cover $\{P_n\}$ with
  $P_n\subset U_n$, $n\in\N$.  Whenever $x\in X$, set
  $n(x)=\min\{n\in\N: x\in U_n\}$, and define an l.s.c.\ mapping
  $\Phi:X\to \mathscr{L}(E)$ by $\Phi(x)=F_{n(x)}$, $x\in X$. By
  hypothesis, $\Phi$ has a continuous selection $f:X\to E$.  Finally,
  consider the composite mapping
  $\theta=\psi\circ f:X\to \mathscr{C}(H)$, which is clearly
  u.s.c. So, for every $n\in \N$, the set
  $P_n=\theta^{-1}[\{z_1,\dots,z_n\}]=
  f^{-1}\left(\psi^{-1}[\{z_1,\dots,z_n\}]\right)$ is closed in $X$,
  and $P_n\subset U_n$ because
  $\{z_1,\dots z_n\}\subset F_n\setminus F_{n+1}$. Indeed, $x\in P_n$
  implies that $\psi(f(x))\subset \{z_1,\dots,z_n\}$ and, therefore,
  $m(f(x))\leq n$, by the definition of $\varphi$. Accordingly,
  $f(x)\notin F_{n+1}$ and, by the definition of $\Phi$, we get that
  $n(x)\leq n$. Thus, $x\in U_{n(x)}\subset U_n$ and $P_n\subset
  U_n$. Since $\{P_n\}$ is covering $X$, this is impossible.
\end{proof}

The following question is a partial case of a question stated in
\cite{nedev-valov:84}. 

\begin{question}[\cite{nedev-valov:84}]
  \label{question-shsa-v18:1}
  Let $X$ be a collectionwise normal space, $E$ be a Banach space, and
  $\mathscr{L}(E)\subset \mathscr{F}_\mathbf{c}(E)$ be such that
  $\mathscr{C}'_\mathbf{c}(E)\subset \mathscr{L}(E)$ and any
  decreasing sequence of elements of $\mathscr{L}(E)$ has a nonempty
  intersection.  Then, is it true that each l.s.c.\ mapping
  $\Phi:X\to \mathscr{L}(E)$ has a continuous selection?
\end{question}

\begin{remark}
  \label{remark-shsa-vgg-rev:1}
  An elegant alternative proof of Theorem \ref{theorem-res-pro-v2:1}
  was offered by the referee. Namely, take a decreasing sequence
  $\{F_n\}$ of elements $\mathscr{L}(E)$. Since $X$ is normal but not
  countably paracompact, there exists a decreasing sequence $\{P_n\}$
  of closed subsets of $X$ such that $P_1=X$,
  $\bigcap_{n=1}^\infty P_n =\emptyset$ and
  $\bigcap_{n=1}^\infty U_n\neq\emptyset$, for every sequence
  $\{U_n\}$ of open subsets of $X$ with $P_n\subset U_n$, $n\in\N$,
  see \cite[Corollary 5.2.2]{engelking:89}. Next, for every $x\in X$,
  let $n(x)=\max\{n\in\N: x\in P_n\}$ which is a well-defined element
  of $\N$ because $\bigcap_{n=1}^\infty P_n =\emptyset$. Finally,
  define $\Phi:X\to \mathscr{L}(E)$ by ${\Phi(x)=F_{n(x)}}$, $x\in
  X$. Then $\Phi$ is l.s.c.\ and by hypothesis, it has a continuous
  selection $f:X\to E$. Since $\bigcap_{n=1}^\infty F_n$ is closed and
  $P_n\subset f^{-1}\left(\mathbf{O}_{1/n}(F_n)\right)$, for every
  $n\in\N$, we get that
  \[
    f^{-1}\left(\bigcap_{n=1}^\infty F_n\right)=
  f^{-1}\left(\bigcap_{n=1}^\infty\mathbf{O}_{1/n}(F_n)\right)=
  \bigcap_{n=1}^\infty
  f^{-1}\left(\mathbf{O}_{1/n}(F_n)\right)\neq\emptyset.
  \]
  Accordingly, $\bigcap_{n=1}^\infty F_n\neq\emptyset$ as
  required.\hfill\textsquare
\end{remark}

\subsection{Selections and Reflexive Banach Spaces}

The property in Theorem \ref{theorem-res-pro-v2:1} has the following
natural interpretation in the setting of Banach spaces. 

\begin{theorem}[\cite{MR0002006}]
  \label{theorem-shsa-v19:1}
  A normed space $E$ is a reflexive Banach space if and only if
  every decreasing sequence of nonempty closed bounded
  convex subsets of $E$ has a nonempty intersection. 
\end{theorem}

For reflexive Banach spaces, the following interesting result was
obtained by Stoyan Nedev \cite{nedev:87}.

\begin{theorem}[\cite{nedev:87}]
  \label{theorem-shsa-v19:2}
  Whenever $E$ is a reflexive Banach space,  each l.s.c.\ mapping
  ${\Phi:\omega_1\to \mathscr{F}_\mathbf{c}(E)}$ has a continuous
  selection.
\end{theorem}

Here, $\omega_1$ is the first uncountable ordinal endowed with the
order topology. This result was further generalised in
\cite{MR1606612} by replacing $\omega_1$ with an arbitrary
suborderable space.

\begin{theorem}[\cite{MR1606612}]
  \label{theorem-shsa-v19:3}
  If $X$ is a suborderable space and $E$ is a reflexive Banach space,
  then each l.s.c.\ mapping ${\Phi:X\to \mathscr{F}_\mathbf{c}(E)}$
  has a continuous selection.
\end{theorem}

\emph{Suborderable spaces} are precisely the subspaces of orderable
spaces, and are also called \emph{generalised ordered}. Every
suborderable space is countably paracompact and collectionwise normal,
but not necessarily paracompact. For instance, $\omega_1$ is not
paracompact. Based on this, the following general question was stated
in \cite{MR1606612}; it is known as Choban-Gutev-Nedev conjecture.

\begin{question}[\cite{MR1606612}]
  \label{question-shsa-v19:1}
  Let $X$ be a countably paracompact and collectionwise normal space,
  $E$ be a Hilbert (or reflexive Banach) space and
  $\Phi:X\to \mathscr{F}_\mathbf{c}(E)$ be an l.s.c.\ mapping. Does
  $\Phi$ have a continuous selection?
\end{question}

\subsection{Norm-Minimal Selections}

Let $E$ be a normed linear space. A selection $f:X\to E$ for a mapping
$\Phi:X\sto E$ is called \emph{minimal} with respect to the norm
$\|\cdot\|$ of $E$, or \emph{norm-minimal}, see
\cite{aubin-frankowska:90}, if
\begin{equation}
  \label{eq:shsa-v21:1}
\|f(x)\|=\min\big\{\|y\|:y\in\Phi(x)\big\},\quad\text{for every $x\in
X$}.
\end{equation}
A norm $\|.\|$ on $E$ is called \emph{locally uniformly rotund},
abbreviated \emph{LUR}, if for each $y\in E$ and sequence
$\{y_n\}\subset E$,
\begin{equation}
  \label{eq:shsa-v21:2}
  \begin{cases}
  \lim_{n\to \infty}\|y_n\|=\|y\|,\ \ \text{and}\\
 \lim_{n\to\infty}\|y_n+y\|=2\|y\|,
\end{cases}
\quad\text{implies}\quad \lim_{n\to\infty}\|y_n-y\|=0.
\end{equation}

If $E$ is a normed space equipped with an LUR norm, then every
nonempty closed convex subset of $ E$ has a unique point with a
minimal norm, see \cite[Lemma 4.1]{MR1606612}. Accordingly, we have
the following observation.

\begin{proposition}
  \label{proposition-shsa-v21:1}
  Let $E$ be a normed space equipped with an LUR norm. Then each
  mapping $\Phi:X\to \mathscr{F}_\mathbf{c}(E)$ has a unique
  norm-minimal selection.
\end{proposition}

Regarding continuity of norm-minimal selections, the following
characterisation was obtained in \cite[Theorem 4.1]{Gutev2001}. In
this theorem, $\B$ is the closed unit ball of a normed space $E$
equipped with a norm $\|\cdot\|$.

\begin{theorem}[\cite{Gutev2001}]
  \label{theorem-shsa-v21:2}
  Let $X$ be a space and $E$ be a normed space equipped with an LUR
  norm. Then for an l.s.c.\ mapping
  $\Phi:X\to \mathscr{F}_\mathbf{c}(E)$, the following two conditions
  are equivalent\textup{:}
  \begin{enumerate}
  \item $\Phi^{-1}[\varepsilon\B]$ is closed in $X$, for every
    $\varepsilon>0$.
  \item $\Phi$ admits a continuous norm-minimal selection.
  \end{enumerate}
\end{theorem}
  
As for normed spaces which admit an equivalent LUR norm, let us
explicitly state the famous Troyanski's renorming theorem
\cite[Theorem 1]{troyanski:71}.

\begin{theorem}[\cite{troyanski:71}]
  \label{theorem-shsa-v21:3}
  Every reflexive Banach space admits a topologically equivalent LUR
  norm.
\end{theorem}

Based on this theorem and norm-minimal selections, the following
interesting result about Question \ref{question-shsa-v19:1} was
obtained by Shishkov \cite[Proposition 1.1]{MR2071296}.

\begin{theorem}[\cite{MR2071296}]
  \label{theorem-shsa-v21:4}
  Let $E$ be a reflexive Banach space and $X$ be a space such that
  each weak $\theta$-cover of $X$ has an open locally finite
  refinement. Then every l.s.c.\ mapping
  $\Phi: X\to \mathscr{F}_\mathbf{c}(E)$ has a continuous selection.
\end{theorem}

Here, a cover $\mathscr{U}$ of $X$ is called a \emph{weak
  $\theta$-cover} if $\mathscr{U}$ is a countable union of open
families $\mathscr{U}_k$, $k\in \N$, such that for each $x\in X$,
there exists some $k(x)\in \N$ for which the family
$\mathscr{U}_{k(x)}$ has a positive finite order at $x$, namely
\[
  0<\left|\left\{ U\in \mathscr{U}_{k(x)}: x\in
      U\right\}\right|<+\infty.
\]

The following example was given in the same paper of Shishkov, see
\cite[Theorem 1.2]{MR2071296}.

\begin{example}[\cite{MR2071296}]
  \label{example-shsa-v21:1}
  If $X$ is a countably paracompact and hereditarily collectionwise
  normal space, then each weak $\theta$-cover of $X$ has an open
  locally finite refinement. 
\end{example}

Example \ref{example-shsa-v21:1} implies that Theorem
\ref{theorem-shsa-v21:4} is a natural generalisation of Theorem
\ref{theorem-shsa-v19:3} because each suborderable space is countably
paracompact and hereditarily collectionwise normal. In fact, Theorem
\ref{theorem-shsa-v21:4} is a potential candidate for the affirmative
solution of Question \ref{question-shsa-v19:1} in view of the
following characterisation of countably paracompact collectionwise
normal spaces claimed in \cite{Smith1980}.

\begin{theorem}[\cite{Smith1980}]
  \label{theorem-shsa-v21:5}
  A space $X$ is countably paracompact and collectionwise normal if
  and only if each weak $\theta$-cover of $X$ has an open locally
  finite refinement.
\end{theorem}

However, as pointed out in \cite{MR2071296}, the proof of Theorem
\ref{theorem-shsa-v21:5} in \cite{Smith1980} is incomplete, which
suggests the following separate question.

\begin{question}
  \label{question-shsa-v21:2}
  Let $X$ be a countably paracompact collectionwise normal space and
  $\mathscr{U}$ be a weak $\theta$-cover of $X$. Is it true that
  $\mathscr{U}$ has an open locally finite refinement?
\end{question}

\subsection{Selections and Hilbert Spaces}

In case of Hilbert spaces, Question \ref{question-shsa-v19:1} was
resolved in the affirmative by Ivailo Shishkov in 2005, his paper with
the final solution appeared in print in  \cite{MR2406397}.

\begin{theorem}[\cite{MR2406397}]
  \label{theorem-contr-ext-14:10}
  A space $X$ is countably paracompact and collectionwise normal iff
  for every Hilbert space $E$, every l.s.c.\ mapping
  ${\Phi : X\to \mathscr{F}_\mathbf{c}(E)}$ has a continuous
  selection.
\end{theorem}

The role of countable paracompactness in Theorem
\ref{theorem-contr-ext-14:10} is the equivalence stated in Proposition
\ref{proposition-shsa-v12:1}, while the essential selection property
of collectionwise normality was obtained in \cite[Theorem
1.3]{MR2406397}.

\begin{theorem}[\cite{MR2406397}]
  \label{theorem-shsa-v20:1}
  If $X$ is a collectionwise normal space and $E$ is a Hilbert space,
  then each bounded l.s.c.\ mapping $\Phi : X\to
  \mathscr{F}_\mathbf{c}(E)$ has a continuous selection.
\end{theorem}

Here are two consequences of Theorem \ref{theorem-shsa-v20:1}, which
may shed some light on the role of Hilbert spaces in the selection
problem for collectionwise normal spaces. The first one shows that the
answer to Question \ref{question-shsa-v14:1} is ``Yes'' provided the
range $E$ is a Hilbert space.

\begin{corollary}
  \label{corollary-shsa-v20:1}
  Let $X$ be a collectionwise normal space, $E$ be a Hilbert space
  and $\Phi:X\to \mathscr{F}_\mathbf{c}(E)$ be an l.s.c.\ mapping. If
  $\Phi$ has a continuous $\varepsilon$-selection for some
  $\varepsilon>0$, then $\Phi$ also has a continuous selection.  
\end{corollary}

\begin{proof}
  Let $g:X\to E$ be a continuous $\varepsilon$-selection for $\Phi$,
  for some $\varepsilon>0$. Just as in Remark \ref{remark-shsa-v18:1},
  consider the l.s.c.\ mapping
  $\overline{\Phi\wedge \mathbf{O}_\varepsilon[g]}:X\to
  \mathscr{F}_\mathbf{c}(E)$ which is bounded-valued. Hence, by
  Theorem \ref{theorem-shsa-v20:1}, it has a continuous selection
  ${f:X\to E}$. Evidently, $f$ is also a selection for $\Phi$.
\end{proof}

The other consequence should be compared with the characterisation of
PF-normality in Theorem \ref{theorem-res-pro-v1:3}.

\begin{corollary}
  \label{corollary-shsa-v20:2}
  Let $X$ be a space such that for every Hilbert space $E$, every
  bounded l.s.c.\ mapping $\Phi:X\to \mathscr{F}_\mathbf{c}(E)$ has a
  continuous selection. Then $X$ is collectionwise normal.
\end{corollary}

\begin{proof}
  Let $\mathscr{D}$ be a discrete collection of nonempty closed
  subsets of $X$, and $\ell_2 (\mathscr{D})$ be the Hilbert space of
  all functions $y : \mathscr{D} \to \R$ with
  ${\sum _{D \in \mathscr{D}} [y (D)]^2 < \infty}$, where the linear
  operations are defined pointwise and the norm is
  $\| y \| = \sqrt{\sum _{D \in \mathscr{D}} [y (D)]^2} $. Set
  $A=\bigcup \mathscr{D}$, and define a map
  $g:A\to \ell_2(\mathscr{D})$ by
  $g(x)=\chi_D:\mathscr{D}\to \{0,1\}\subset \R$ to be the
  characteristic function of the unique $D \in \mathscr{D}$ with
  $x\in D$. Since $\mathscr{D}$ is closed and discrete, $A$ is a
  closed subset of $X$ and $g$ is continuous. Moreover, $g$ takes
  values in the closed unit ball $\B$ of $\ell_2(\mathscr{D})$ because
  $\|\chi_D\|=1$, for each $D\in \mathscr{D}$. Finally, let
  $\Phi_g:X\to \mathscr{F}_\mathbf{c}(\B)\subset
  \mathscr{F}_\mathbf{c}(\ell_2(\mathscr{D}))$ be as in Proposition
  \ref{proposition-contr-ext-14:1} with $Y$ replaced by $\B$.  Then
  $\Phi_g$ is both l.s.c.\ and is bounded-valued. Hence, by
  hypothesis, it has a continuous selection
  $f:X\to \ell_2(\mathscr{D})$, which is a continuous extension of
  $g$, by Proposition \ref{proposition-contr-ext-14:1}. Since
  $\{\chi_D:D\in \mathscr{D}\}$ is a closed discrete set in
  $\ell_2(\mathscr{D})$ (actually, uniformly discrete with respect to
  the norm), there exists a discrete collection
  $\{\mathscr{U}_D:D\in \mathscr{D}\}$ of open subsets of
  $\ell_2(\mathscr{D})$ with $\chi_D\in \mathscr{U}_D$, for each
  $D\in \mathscr{D}$. Then $U_D=f^{-1}(\mathscr{U}_D)$,
  $D\in \mathscr{D}$, is a discrete collection of open subsets of $X$
  such that $D\subset U_D$, for each $D\in \mathscr{D}$. Accordingly,
  $X$ is collectionwise normal.
\end{proof}

\subsection{Selections and $\ell_p(\mathscr{A})$-spaces}

For a set $\mathscr{A}$ and $p\geq 1$, let $\ell_p (\mathscr{A})$ be
the Banach space of all functions $y: \mathscr{A} \to \R$ with
${\sum _{\alpha \in \mathscr{A}} |p (\alpha)|^p < \infty}$, where the
linear operations are defined pointwise and the norm is
$\| y \|_p = \left(\sum _{\alpha \in \mathscr{A}} |y(\alpha)|^p
\right)^{\frac1p}$. Since every Hilbert space is isomorphic to
$\ell_2(\mathscr{A})$ for some set $\mathscr{A}$, Theorem 
\ref{theorem-contr-ext-14:10} can be restated in the following terms.  

\begin{theorem}[\cite{MR2406397}]
  \label{theorem-shsa-v21:1}
  For a space $X$, the following are equivalent\textup{:}
  \begin{enumerate}
  \item $X$ is countably paracompact and collectionwise normal.
  \item For every set $\mathscr{A}$, every l.s.c.\ mapping
    $\Phi : X\to \mathscr{F}_\mathbf{c}(\ell_2(\mathscr{A}))$ has a
    continuous selection.
  \end{enumerate}
\end{theorem}

Similarly, the characterisation of paracompactness in
Theorem \ref{theorem-res-pro-v9:1} can be restated in terms of the
Banach space $\ell_1(\mathscr{A})$. The following theorem is actually
reassembling the proof of Theorem  \ref{theorem-res-pro-v9:1}.

\begin{theorem}[\cite{michael:56a}]
  \label{theorem-contr-ext-14:9}
  For a space $X$, the following are equivalent\textup{:}
  \begin{enumerate}
  \item\label{item:shsa-vgg:1} Every open cover $\mathscr{U}$ of $X$ has a
    locally finite open refinement.
  \item\label{item:shsa-vgg:2} For every set $\mathscr{U}$, every l.s.c.\
    mapping $\Phi:X\to \mathscr{F}_\mathbf{c}(\ell_1(\mathscr{U}))$
    has a continuous selection.
  \item\label{item:shsa-vgg:3} Every open cover $\mathscr{U}$ of $X$
    has an index-subordinated partition of unity.
  \end{enumerate}
\end{theorem}

Here, a collection $\xi_U:X\to [0,1]$, $U\in \mathscr{U}$, of
continuous functions on a space $X$ is a \emph{partition of unity} if
$\sum_{U\in \mathscr{U}}\xi_U(x)=1$, for each $x\in X$. A partition of
unity $\{\xi_U:U\in \mathscr{U}\}$ is \emph{index-subordinated} to a
cover $\mathscr{U}$ of $X$ if $X\setminus U\subset \xi_U^{-1}(0)$, for
each $U\in \mathscr{U}$.  The following natural result may explain the
relationship between partitions of unity, index-subordinated to open
covers $\mathscr{U}$ of $X$, and continuous selections
$f:X\to \ell_1(\mathscr{U})$, see Theorem
\ref{theorem-contr-ext-14:9}. The result itself was obtained in the
implication (f)$~\Rightarrow~$(a) of \cite[Theorem 1.2]{morita:60} and
was explicitly stated in \cite[Proposition 5.4]{MR1425941}; the case
of locally finite partitions of unity was obtained in the proof of
\cite[Theorem 1]{MR597065}.

\begin{lemma}
  \label{lemma-Count-Par-v12:2}
  Let $X$ be a space and $\xi_\alpha:X\to [0,1]$,
  $\alpha\in \mathscr{A}$, be a collection of functions. Then
  $\{\xi_\alpha:\alpha\in \mathscr{A}\}$ is a partition of unity on
  $X$ if and only if the diagonal map
  $\xi=\Delta_{\alpha\in\mathscr{A}}\xi_\alpha:X\to
  \ell_1(\mathscr{A})$ is continuous and satisfies $\|\xi(x)\|_1=1$,
  for every $x\in X$.
\end{lemma}

Theorems \ref{theorem-shsa-v21:1} and
\ref{theorem-contr-ext-14:9} reveal an interesting role of
the spaces $\ell_1(\mathscr{A})$ and $\ell_2(\mathscr{A})$ in the
selection theory. In this regard, let us recall that all Banach spaces
$\ell_p(\mathscr{A})$, $1\leq p<+\infty$, are homeomorphic, and it is
well known that $\ell_1(\mathscr{A})$ is not reflexive, but each
$\ell_p(\mathscr{A})$, $1<p<+\infty$, is reflexive. In fact, each
infinite-dimensional reflexive Banach space is homeomorphic to
$\ell_1(\mathscr{A})$, for some $\mathscr{A}$, see e.g.\ \cite[Theorem
4.1 in \S 4 of Chapter VII]{bessage-pelczynski:75}. This brings the
following natural partial case of Question \ref{question-shsa-v19:1}.

\begin{question}
  \label{question-shsa-v21:1}
  Let $X$ be a countably paracompact collectionwise normal space,
  and $\Phi:X\to \mathscr{F}_\mathbf{c}(\ell_p(\mathscr{A}))$ be an
  l.s.c.\ mapping for some $\mathscr{A}$ and $1<p<+\infty$ with
  $p\neq 2$.  Does $\Phi$ have a continuous selection?
\end{question}

Going back to Question \ref{question-shsa-v19:1}, the interested
reader may consult some of the papers of Shishkov
(\cite{MR1779519,MR1825384,MR2076092,MR2071296,MR2350718}), which
contain several interesting ideas.

\section{Selections and Finite-Dimensional Spaces}

\subsection{Selection Extension and Approximation Properties}

Let $n\geq-1$. A family $\mathscr{S}$ of subsets of a space $Y$ is
\emph{equi-$LC^n$} \cite{michael:56b} if every neighbourhood $U$ of a
point $y\in \bigcup\mathscr{S}$ contains a neighbourhood $V$ of $y$
such that for every $S\in \mathscr{S}$, every continuous map
$g:\s^k\to V\cap S$ of the $k$-sphere $\s^k$, $k\leq n$, can be
extended to a continuous map $h:\B^{k+1}\to U\cap S$ of the
$(k+1)$-ball $\B^{k+1}$.  A space $S$ is  $C^n$ if for every
$k\leq n$, every continuous map $g:\s^k\to S$ can be extended to a
continuous map $h:\B^{k+1}\to S$. In these terms, a family
$\mathscr{S}$ of subsets of $Y$ is equi-$LC^{-1}$ if it consists of
nonempty subsets; similarly, each nonempty subset $S\subset Y$ is
$C^{-1}$.\medskip

A mapping $\Phi:X\sto Y$ has the \emph{Selection Extension Property}
(or \emph{SEP}) \emph{at} a closed subset $A\subset X$
\cite{michael:80} if every continuous selection $g:A\to Y$ for
$\Phi\uhr A$ can be extended to a continuous selection for $\Phi$. If
$g$ only extends to a continuous selection for $\Phi\uhr U$ for some
neighbourhood $U$ of $A$ in $X$, then $\Phi$ is said to have the
\emph{Selection Neighbourhood Extension Property} (or \emph{SNEP})
\emph{at} $A$ \cite{michael:80}. If this holds for any closed set
of $X$, then we simply say that $\Phi$ has the SEP, or the
SNEP. \medskip

For a subset $Z\subset X$, we write $\dim_X(Z)\leq m$ to express that
the covering dimension $\dim(S)\leq m$, for every $S\subset Z$ which
is closed in $X$, see \cite{michael:56b}. Let us remark that for a
normal space $X$, $\dim_X(Z)\leq m$ is valid if either $\dim(Z)\leq m$
or $\dim(X)\leq m$.  The following theorem was obtained in
\cite[Theorem 1.2]{michael:56b} and is commonly called the
\emph{finite-dimensional selection theorem}.

\begin{theorem}[\cite{michael:56b}]
  \label{theorem-shsa-v21:6}
  Let $X$ be a paracompact space, $A\subset X$ be a closed set with
  $\dim_X(X\setminus A)\leq n+1$, $Y$ be a completely metrizable
  space and $\mathscr{S}\subset \mathscr{F}(Y)$ be an equi-$LC^n$
  family. Then every l.s.c.\ mapping $\Phi:X\to \mathscr{S}$ has the
  SNEP at $A$.  If, moreover, each $S\in \mathscr{S}$ is $C^n$, then
  $\Phi$ also has the SEP at $A$.
\end{theorem}

The proof of Theorem \ref{theorem-shsa-v21:6} in \cite{michael:56b} is
based on a uniform version of the same theorem. To this end, let us
recall that a family $\mathscr{S}$ of subsets of a metric space
$(Y,\rho)$ is \emph{uniformly equi-$LC^{n}$}, where $n\geq -1$, if for
every $\varepsilon>0$ there exists $\delta(\varepsilon)>0$ such that,
for every $S\in\mathscr{S}$, every continuous map of the $k$-sphere
($k\leq n$) in $S$ of diameter$\ < \delta(\varepsilon)$ can be
extended to continuous map of the $(k+1)$-ball into a subset of $S$ of
diameter$\ <\varepsilon$ \cite{michael:56b}. The relation with
equi-$LC^n$ families is given by the following embedding property
stated in \cite[Theorem 3.1]{michael:56b}, see also \cite[Proposition
2.1]{michael:56b} and \cite[Theorem 1]{dugundji-michael:56}.

\begin{theorem}[\cite{michael:56b}]
  \label{theorem-shsa-v22:1}
  Let $\mathscr{S}\subset \mathscr{F}(Y)$ be an equi-$LC^n$ family of
  subsets of a completely metrizable space $Y$. Then
  $\bigcup\mathscr{S}$ can be embedded into a Banach space $E$ so that
  $\mathscr{S}\subset \mathscr{F}(E)$ is a uniformly equi-$LC^n$
  family of subsets of $E$.
\end{theorem}

The other reduction in the proof of Theorem \ref{theorem-shsa-v21:6}
is that the properties ``SNEP'' and ``SEP'' are obtained by the
following uniform selection approximation property \cite[Theorem
4.1]{michael:56b}.

\begin{theorem}[\cite{michael:56b}]
  \label{theorem-shsa-v21:7}
  Let $(Y,\rho)$ be a complete metric space and
  $\mathscr{S}\subset \mathscr{F}(Y)$ be uniformly equi-$LC^n$. Then
  to every $\varepsilon > 0$ there corresponds
  $\gamma(\varepsilon) > 0$ with the following property\textup{:} If
  $\Phi:X\to \mathscr{S}$ is an l.s.c.\ mapping from a paracompact
  space $X$ with $\dim(X)\leq n+1$, then for every continuous
  $\gamma(\varepsilon)$-selection $g:X\to Y$ for $\Phi$, there exists
  a continuous selection $f:X\to Y$ for $\Phi$ such that
  $\rho(f(x),g(x))<\varepsilon$, for all $x\in X$.  Moreover, if each
  $S\in \mathscr{S}$ is $C^n$, then one can take
  $\gamma(+\infty)=+\infty$.
\end{theorem}

\subsection{Hybrid Selection Theorems}

As a common generalisation of Theorems \ref{theorem-res-pro-v9:1} and
\ref{theorem-shsa-v21:6}, the following two theorems were obtained in
\cite{michael:80}. 

\begin{theorem}[\cite{michael:80}]
  \label{theorem-shsa-v21:8}
  Let $X$ be a paracompact space, $E$ be a Banach space, $Z\subset X$
  with $\dim_X(Z)\leq n + 1$, and $\Phi:X \to \mathscr{F}(E)$ be an
  l.s.c.\ mapping with $\Phi(x)$ convex for all $x\in X\setminus Z$,
  and with $\{\Phi(x):x\in Z\}$ uniformly equi-$LC^n$. Then
  $\Phi$ has the SNEP. If, moreover, $\Phi(x)$ is $C^n$ for
  every $x\in Z$, then $\Phi$ has the SEP.
\end{theorem}

To state the other theorem, let us recall that a family $\mathscr{S}$
of subsets of a space $Y$ is \emph{equi-$LC^n$ in $Y$}
\cite{michael:80} if every neighbourhood $U$ of a point $y\in Y$
contains a neighbourhood $V$ of $y$ such that for every
$S\in \mathscr{S}$, every continuous $g:\s^k\to V\cap S$, for
$k\leq n$, can be extended to a continuous $h:\B^{k+1}\to U\cap
S$. Each family of subsets of $Y$ which is equi-$LC^n$ in $Y$ is also
equi-$LC^n$, but the converse is not necessarily true. Here is a
simple example.

\begin{example}
  For every positive real number $t>0$, let
  \[
    S_t=\{x\in\R: |x|\geq t\}=(-\infty,-t]\cup[t,+\infty).
  \]
  Then $\mathscr{S}=\{S_t:t>0\}$ is equi-$LC^n$ for all $n\geq-1$, but
  is not equi-$LC^0$ in $\R$.\hfill\textsquare
      
\end{example}

Finally, let us also recall that a metrizable space $Y$ is an
\emph{AR} (respectively, \emph{ANR}) if it is a retract (respectively,
neighbourhood retract) of every metric space $E$ containing it as a
closed subset.

\begin{theorem}[\cite{michael:80}]
  \label{theorem-shsa-v21:9}
  Let $X$ be a paracompact space, $Y$ be a completely metrizable ANR,
  $Z\subset X$ with $\dim_X(Z)\leq n + 1$, and
  $\Phi:X \to \mathscr{F}(Y)$ be an l.s.c.\ mapping with $\Phi(x)=Y$
  for all $x\in X\setminus Z$, and with $\{\Phi(x):x\in Z\}$
  equi-$LC^n$ in $Y$. Then $\Phi$ has the SNEP. If, moreover, $Y$ is
  an AR and $\Phi(x)$ is $C^n$ for every $x\in Z$, then $\Phi$ has the
  SEP.
\end{theorem}

We proceed with an example showing that Theorem
\ref{theorem-shsa-v21:8} fails if in this theorem the collection
$\{\Phi(x): x\in Z\}$ is assumed to be only equi-$LC^n$.

\begin{example}
  \label{example-shsa-v22:1}
  Let $\D\subset \R^2$ be the closed unit disk in $\R^2$, and $\s$ be
  the unit circle. Also, let
  $S_y=\{(s,t): s^2+t^2=1\ \text{and}\ t\geq y\}$, for every
  $-1<y\leq 1$. Define $\Phi:\D\to \mathscr{F}(\R^2)$ by letting for
  $(x,y)\in\D$ that
  \[
    \Phi(x,y)=
    \begin{cases}
      \big\{(x,y)\big\} &\text{if $(x,y)\in\s$,}\\
      S_y &\text{if $(x,y)\notin \s$.}
    \end{cases}
  \]
  Then $\Phi$ is l.s.c., but has no continuous selection because each
  such selection will be a retraction $r:\D\to \s$. However, each
  $\Phi(x,y)$, $(x,y)\in \s$, is convex being a singleton. Moreover,
  the collection $\mathscr{S}=\{S_y: -1<y\leq 1\}$ of arcs is
  equi-$LC^n$ for every $n\geq -1$, and each element of $\mathscr{S}$
  is $C^n$. Finally, we also have that
  $\dim_{\D}(\D\setminus \s)=\dim(\D\setminus \s)=2$.\hfill\textsquare
\end{example}

The case when the family in Theorem \ref{theorem-shsa-v21:8} is
assumed to be equi-$LC^n$ in $E$ is not covered by this example, which
brings the following question.

\begin{question}
  \label{question-shsa-v22:1}
  Is Theorem \ref{theorem-shsa-v21:8} still valid if in this theorem
  ``uniformly equi-$LC^n$'' is replaced by ``equi-$LC^n$ in $E$''?
\end{question}

This also brings a similar question about Theorem
\ref{theorem-shsa-v21:9} of whether this theorem is still valid if
``equi-$LC^n$ in $Y$'' is replaced by ``equi-$LC^n$''. This doesn't
seem likely, but is not covered by Example
\ref{example-shsa-v22:1}.\medskip

Let us remark that in the special case of $n=-1$, Theorem
\ref{theorem-shsa-v21:9} is covered by Theorem
\ref{theorem-shsa-v21:8}, see the remark after the proof of Theorem
\ref{theorem-shsa-v21:9} in \cite{michael:80}. Moreover, if $Y$ and
$\Phi$ are as in Theorem \ref{theorem-shsa-v21:9}, then by
\cite[Lemma 6.1]{michael:80}, $Y$ can be embedded as a closed subset
in a Banach space $E$ such that
\begin{enumerate}[label=(\roman*)]
\item $Y$ is a \emph{uniform ANR} (respectively, \emph{uniform AR}) of
  $E$, and
\item $\{\Phi(x): x\in Z\}$ is uniformly equi-$LC^n$ in $E$.
\end{enumerate}
Here, a closed subset $Y\subset E$ is a \emph{uniform ANR} of $E$
\cite{michael:79} if to every $\varepsilon > 0$ corresponds some
$\delta(\varepsilon) > 0$ and a retraction
$r: \mathbf{O}_{\delta(\infty)}(Y)\to Y$ such that
${\|z- r(z)\| < \varepsilon}$, whenever
$z\in \mathbf{O}_{\delta(\varepsilon)}(Y)$. If one can take
$\delta(\infty) = \infty$ (so that the domain of $r$ is always $E$),
then $Y$ is called a \emph{uniform AR} of $E$. Accordingly, Theorem
\ref{theorem-shsa-v21:9} can be reformulated in the following way.

\begin{theorem}
  \label{theorem-shsa-v22:2}
  Let $X$ be a paracompact space, $E$ be a Banach space, $Y\subset E$
  be a closed subset of $E$ which is a uniform ANR of $E$,
  $Z\subset X$ with $\dim_X(Z)\leq n + 1$, and
  $\Phi:X \to \mathscr{F}(Y)$ be an l.s.c.\ mapping with $\Phi(x)=Y$
  for all $x\in X\setminus Z$, and with $\{\Phi(x):x\in Z\}$ uniformly
  equi-$LC^n$. Then $\Phi$ has the SNEP. If, moreover, $Y$ is a
  uniform AR of $E$ and $\Phi(x)$ is $C^n$ for every $x\in Z$, then
  $\Phi$ has the SEP.
\end{theorem}

This makes Theorems \ref{theorem-shsa-v21:8} and
\ref{theorem-shsa-v21:9} further similar in the following sense. If
$\Phi$ is as in Theorem \ref{theorem-shsa-v21:8}, then the l.s.c.\
mapping
$\tilde{\Phi}=\overline{\conv[\Phi]}:X\to \mathscr{F}_\mathbf{c}(E)$
has the SEP, see Theorem \ref{theorem-res-pro-v5:1}. If $\Phi$ is as
in Theorem \ref{theorem-shsa-v22:2}, then the constant mapping
$\tilde{\Phi}(x)=Y$, $x\in X$, has the SNEP, and also the SEP provided
$Y$ is a uniform AR of $E$. Moreover, in both cases, the pair
$(\Phi,\tilde{\Phi})$ of mappings has the following properties:
\begin{gather}
  \label{eq:shsa-v22:3}
  \Phi\ \text{is an l.s.c.\ set-valued selection for $\tilde{\Phi}$;}\\
  \label{eq:shsa-v22:1}
  \dim_X\big(\{x\in X: \Phi(x)\neq \tilde{\Phi}(x)\}\big)\leq n+1;\\
  \label{eq:shsa-v22:2}
  \big\{\Phi(x):x\in X\ \text{and}\ \Phi(x)\neq \tilde{\Phi}(x)\big\}\
  \text{is uniformly equi-$LC^n$ in $E$.}
\end{gather}

This motivates the following further question.

\begin{question}
  Let $X$ be a paracompact space and $E$ be a Banach space. Suppose
  that ${(\Phi,\tilde{\Phi}):X\to \mathscr{F}(E)}$ is a pair of
  mappings as in \eqref{eq:shsa-v22:3}, \eqref{eq:shsa-v22:1} and
  \eqref{eq:shsa-v22:2}. Does $\Phi$ have the SNEP provided so does
  $\tilde{\Phi}$?  Similarly, does $\Phi$ have the SEP provided
  $\tilde{\Phi}$ has the SEP and $\Phi(x)$ is $C^n$, for every
  $x\in X$ with $\Phi(x)\neq \tilde{\Phi}(x)$?
\end{question}

\subsection{Selection Homotopy Extension Properties}

The following selection interpretation of the Borsuk homotopy
extension theorem \cite{borsuk:36} was obtained by Michael, see 
\cite[Theorem 3.4]{michael:57}.

\begin{theorem}[\cite{michael:57}]
  \label{theorem-shsa-v25:1}
  Let $X$ be a paracompact space, $A\subset X$ be a closed set with
  $\dim_X(X\setminus A)\leq n$, $(Y,\rho)$ be a complete metric space,
  and $\Phi:X\times[0,1]\to \mathscr{F}(Y)$ be a quasi-continuous
  mapping such that $\{\Phi(p):p\in X\times[0,1]\}$ is uniformly
  equi-$LC^n$. Then $\Phi$ has the SEP at
  $ X\times \{0\} \cup A\times [0,1]$.
\end{theorem}

Here, a mapping $\Phi:X\times[0,1]\sto Y$, into a metric space
$(Y,\rho)$, is \emph{quasi-continuous} if it is l.s.c.\ and for every
$\varepsilon>0$, each point of $X\times[0,1]$ has a neighbourhood $U$
such that $\Phi(x,s)\subset \mathbf{O}_\varepsilon(\Phi(x,t))$,
whenever $(x,s),(x,t)\in U$ with $s\leq t$. The interested reader is
referred to \cite{gutev:97b,gutev:00c,Gutev2002}, where Theorem
\ref{theorem-shsa-v25:1} was refined and generalised in various
directions.\medskip

Let $\Phi:X\times[0,1]\sto Y$. A mapping $H:X\times [0,1]\to Y$ is a
\emph{$\Phi$-homotopy} if it is a continuous selection for $\Phi$, and
$H$ is a \emph{$\Phi$-homotopy} of $f:X\times\{0\}\to Y$ if it is also
a continuous extension of $f$. The mapping $\Phi:X\times [0,1]\sto Y$
is said to have the \emph{Selection Homotopy Extension Property} at a
subset $A\subset X$, or the \emph{SHEP} at $A$, if whenever
$f:X\times\{0\}\to Y$ is a continuous selection for
$\Phi\uhr X\times\{0\}$, every $\Phi\uhr A\times [0,1]$-homotopy
$G:A\times [0,1]\to Y$ of $f\uhr A\times\{0\}$ can be extended to a
$\Phi$-homotopy $H:X\times [0,1]\to Y$ of $f$. For instance, the
mapping $\Phi$ in Theorem \ref{theorem-shsa-v25:1} has the SHEP at
$A$.  \medskip

Theorem \ref{theorem-shsa-v25:1} has a nice interpretation for
set-valued mappings $\Phi:X\sto Y$ defined only on $X$. In this
case, we will say that a mapping $H:X\times [0,1]\to Y$ is a
\emph{$\Phi$-homotopy} if $H(x,t)\in \Phi(x)$, for every $x\in X$ and
$t\in [0,1]$. In particular, a selection $f:X\to Y$ will be called
\emph{$\Phi$-homotopic} to a selection $g:X\to Y$ if $f$ and $g$ are
homotopic by a $\Phi$-homotopy $H:X\times[0,1]\to Y$. In these terms,
we have the following consequence of Theorem \ref{theorem-shsa-v25:1}.

\begin{corollary}
  Let $X$ be a paracompact space, $A\subset X$ be a closed set with
  $\dim_X(X\setminus A)\leq n$, $Y$ be a completely metrizable space
  and $\Phi:X\to \mathscr{F}(Y)$ be an l.s.c.\ mapping such that
  $\{\Phi(x):x\in X\setminus A\}$ is equi-$LC^n$ in $Y$. Also, let
  $g,h:A\to Y$ be continuous selections for $\Phi\uhr A$ which are
  $\Phi\uhr A$-homotopic. If one of these selections can be extended
  to a continuous selection for $\Phi$, then so does the other in such
  a way that both selection remain $\Phi$-homotopic.
\end{corollary}

\begin{proof}
  Suppose that $g$ can be extended to a continuous selection
  $f:X\to Y$ for $\Phi$, and take a $\Phi\uhr A$-homotopy
  ${G:A\times [0,1]\to Y}$ between $g$ and $h$, say $G(x,0)=g(x)$ and
  $G(x,1)=h(x)$, for every $x\in A$. For convenience, define a
  continuous map $u:X\times\{0\}\cup A\times[0,1]\to E$ by
  $u(x,t)=G(x,t)$ for $(x,t)\in A\times[0,1]$, and $u(x,0)=f(x)$,
  $x\in X$. Next, define a mapping
  $\Phi_u:X\times[0,1]\to \mathscr{F}(Y)$ by $\Phi_u(x,t)=\{u(x,t)\}$
  if $(x,t)\in X\times\{0\}\cup A\times[0,1]$, and
  $\Phi_u(x,t)=\Phi(x)$ otherwise. Then the family
  $\{\Phi_u(x,t): (x,t)\in X\times[0,1]\}$ remains equi-$LC^n$ in
  $Y$. Hence, $Y$ admits a complete compatible metric $\rho$ so that
  this family is uniformly equi-$LC^n$ with respect to $\rho$, see
  Theorem \ref{theorem-shsa-v22:1}. Moreover, $\Phi_u$ is l.s.c.\
  because so is $\Phi$, see Proposition
  \ref{proposition-shsa-v28:1}. In fact, it is easy to see that
  $\Phi_u$ is quasi-continuous. Thus, by Theorem
  \ref{theorem-shsa-v25:1}, $u$ can be extended to a continuous
  selection $H:X\times [0,1]\to Y$ for $\Phi_u$. This $H$ is a
  $\Phi$-homotopy between $f$ and a continuous extension of the
  selection $h$.
\end{proof}

\subsection{Selections and Weak Deformation Retracts}

A closed subset $A\subset X$ of a space $X$ is called a \emph{weak
  deformation retract} of $X$ if there  exists a continuous
$r:X\times [0,1]\to X$ such that
\[
  \begin{cases}
    r(x,1)=x &\text{for every $x\in X$,}\\
    r(x,0)\in A &\text{for every $x\in X$,}\\
    r(x,0)=x &\text{for every $x\in A$.}
  \end{cases}
\]

The following theorem was proved by Michael \cite[Theorem
6.1]{michael:57}.

\begin{theorem}[\cite{michael:57}]
  \label{theorem-michael-continuous-III}
  Let $X$ be a paracompact space with $\dim(X)\leq n+1$, $(Y,\rho)$ be
  a complete metric space, $\mathscr{S}\subset \mathscr{F}(Y)$ be a
  uniformly equi-$LC^n$ family, $\Phi:X\to \mathscr{S}$ be a
  $\rho$-continuous mapping and $A\subset X$ be a weak deformation
  retract of $X$. Then every continuous selection $g:A\to Y$ for
  $\Phi\uhr A$ can be extended to a continuous selection for $\Phi$.
\end{theorem}

Let us remark that, in contrast to Theorem \ref{theorem-shsa-v21:6},
here there is no requirement that each $\Phi(x)$ is $C^n$. Regarding
the condition $\dim(X)\leq n+1$, the following question was stated by
Michael in \cite{michael:57}.

\begin{question}[\cite{michael:57}]
  \label{question-shsa-v25:1}
  Does Theorem \ref{theorem-michael-continuous-III} remain true if
  $\dim(X)\leq n+1$ is replaced by the weaker requirement that
  $\dim_X(X\setminus A)\leq n+1$?
\end{question}

As commented by Michael, see \cite[Theorem 6.2]{michael:57}, the
answer to Question \ref{question-shsa-v25:1} is ``Yes'' provided the
condition on $A$ is strengthened to the existence of a continuous
$r:X\times [0,1]\to Y$ such that
\[
  \begin{cases}
    r(x,1)=x &\text{for every $x\in X$,}\\
    r(x,0)\in A &\text{for every $x\in X$,}\\
    r(x,t)\in A &\text{for every $x\in A$ and $0\leq t\leq 1$.}
  \end{cases}
\]

\section{Selections and Infinite-Dimensional Spaces}

\subsection{Selections and Countable Dimensionality}

Another hybrid selection theorem representing a common generalisation
of Theorems \ref{theorem-res-pro-v9:1} and \ref{theorem-shsa-v21:6}
was obtained in \cite[Theorem 1.4]{michael:80}. 

\begin{theorem}[\cite{michael:80}]
  \label{theorem-shsa-v21:10}
  Let $X$ be a paracompact space, $A\subset X$ be a closed set with
  $\dim_X(X\setminus A)\leq n+1$, $Z\subset X\setminus A$ with
  $\dim_X(Z)\leq m + 1$, where $m\leq n$, $Y$ be a completely
  metrizable space, and $\Phi:X \to \mathscr{F}(Y)$ be an l.s.c.\
  mapping such that $\{\Phi(x):x\in X\setminus Z\}$ is equi-$LC^n$
  in $Y$ and $\{\Phi(x):x\in Z\}$ is equi-$LC^m$ in $Y$. Then
  $\Phi$ has the SNEP at $A$. If, moreover, $\Phi(x)$ is $C^n$
  for all $x\in X\setminus Z$ and $C^m$ for all $x\in Z$, then
  $\Phi$ has the SEP at $A$.
\end{theorem}

Subsequently, Theorem \ref{theorem-shsa-v21:10} was generalised by
replacing $Z$ with finitely many such sets, see \cite{Ageev1998}. The
case of infinitely many sets seems to offer an interesting
question, which is discussed below. \medskip

The \emph{local dimension} $\ldim(X)$ of a space $X$ was introduced by
Dowker \cite{dowker:55} as the least number $n$ such that each point
of $X$ is contained in an open set $U$ with
$\dim\left(\overline{U}\right)\leq n$. It was shown in
\cite{dowker:55} that $\ldim(X)\leq \dim(X)$ for every normal space
$X$, but there exists a normal space $X$ with $\ldim(X)<\dim(X)$.

\begin{theorem}[\cite{dowker:55}]
  \label{theorem-shsa-v24:1}
  If $X$ is a paracompact space, then $\ldim(X)=\dim(X)$.
\end{theorem}

Subsequently, Wenner \cite{Wenner1972} generalised the local dimension
and introduced the so called locally finite-dimensional spaces. A
space $X$ is \emph{locally finite-dimensional} \cite{Wenner1972} if
each point $p\in X$ has a finite-dimensional neighbourhood. In these
terms, a normal space $X$ is locally finite-dimensional if each point
$p\in X$ is contained in an open set $U\subset X$ with
$\dim\left(\overline{U}\right)<\infty$; equivalently, if each $p\in X$
has a neighbourhood $U\subset X$ with $\dim_X(U)<\infty$.\medskip

A space $X$ is \emph{countable-dimensional} if it is a countable union of
finite-dimensio\-nal subsets \cite{Nagata1959/1960}. A space $X$ is
\emph{strongly countable-dimensional} if it is a countable union of
closed finite-dimensional subsets \cite{Nagata1959/1960}. Each
strongly countable-dimensional space is countable-dimensional, but the
converse is not necessarily true \cite[Example
5.2]{Nagata1959/1960}. Each locally finite-dimensional metrizable
space is strongly countable-dimensional \cite[Theorem
1]{Wenner1972}. Essentially the same proof remains valid for locally
finite-dimensional paracompact spaces.

\begin{theorem}[\cite{Wenner1972}]
  \label{theorem-shsa-v24:2}
  Every  locally finite-dimensional paracompact space is stron\-gly
  countable-dimensional.
\end{theorem}

On the other hand, let us remark that there exists a strongly
countable-dimen\-sional metrizable space which is not locally
finite-dimensional \cite[Theorem 4]{Wenner1972}.\medskip

Regarding selections and locally finite-dimensional spaces, the
following natural ``infinite-dimen\-sional'' version of Theorem
\ref{theorem-shsa-v21:6} can be obtained following the proof of
\cite[Theorem 8.2]{michael:56a}, see also \cite[Theorem
6.2]{michael:88a}. 

\begin{theorem}
  \label{theorem-shsa-v24:3}
  Let $X$ be a locally finite-dimensional paracompact space, $Y$ be a
  completely metrizable space, and $\Phi:X\to \mathscr{F}(Y)$ be an
  l.s.c.\ mapping such that $\{\Phi(x):x\in X\}$ is equi-$LC^n$ for
  each $n\geq -1$. Then $\Phi$ has the SNEP. If, moreover,
  $\Phi(x)$ is $C^n$ for every $x\in X$ and $n\geq-1$, then
  $\Phi$ has the SEP.
\end{theorem}

A space $S$ is called $C^\omega$, or \emph{aspherical}, if every
continuous map $g:\s^k\to S$, $k\geq -1$, can be extended to a
continuous map $h:\B^{k+1}\to S$.  We shall say that a family
$\mathscr{S}\subset \mathscr{F}(Y)$ of subsets of a metric space
$(Y,\rho)$ is \emph{uniformly equi-$LC^\omega$} if for every
$\varepsilon>0$ there exists $\delta(\varepsilon)>0$ such that, for
every $S\in\mathscr{S}$, every continuous map of the $k$-sphere
($k\geq -1$) in $S$ of diameter$\ < \delta(\varepsilon)$ can be
extended to continuous map of the $(k+1)$-ball into a subset of $S$ of
diameter$\ <\varepsilon$.\medskip

In view of Theorems \ref{theorem-shsa-v24:2} and
\ref{theorem-shsa-v24:3}, the following question seems natural.

\begin{question}
  \label{question-shsa-v23:1}
  Let $X$ be a strongly countable-dimensional paracompact space,
  $(Y,\rho)$ be a complete metric space, and
  $\Phi:X\to \mathscr{F}(Y)$ be an l.s.c.\ mapping such that
  $\{\Phi(x):x\in X\}$ is uniformly equi-$LC^\omega$ and each
  $\Phi(x)$, $x\in X$, is $C^\omega$. Does there exist a continuous
  selection for $\Phi$?
\end{question}

Suppose that $X$, $Y$ and $\Phi$ are as in Question
\ref{question-shsa-v23:1}. Then $X=\bigcup_{n=0}^\infty F_n$, for some
increasing sequence of closed sets $F_n\subset X$ with
$\dim(F_n)\leq n$. Thus, inductively, using Theorem
\ref{theorem-shsa-v21:6}, one can construct a selection $f:X\to Y$ for
$\Phi$ such that $f\uhr F_n$ is continuous, for every $n\geq
0$. However, the challenge presented in this question is to make the
construction so that the resulting $f$ will be also continuous.

\subsection{Selections and $C$-spaces}

A space $X$ has property $C$, or $X$ is a
\emph{$C$-space}, \label{page-c-space} if for any sequence
$\{\mathscr{U} _n:n\in\mathbb{N}\}$ of open covers of $X$ there exists
a sequence $\{\mathscr{V} _n:n\in\mathbb{N}\}$ of open
pairwise-disjoint families in $X$ such that each $\mathscr{V} _n$
refines $\mathscr{U} _n$ and $\bigcup_{n\in\N}\mathscr{V} _n$ is a
cover of $X$.  The $C$-space property was originally defined by W.\
Haver \cite{haver:1974} for compact metric spaces, subsequently Addis
and Gresham \cite{addis-gresham:78} reformulated Haver's definition
for arbitrary spaces.  It should be remarked that a $C$-space $X$ is
paracompact if and only if it is countably paracompact and normal, see
e.g.\ \cite[Proposition 1.3]{MR2352366}. Every finite-dimensional
paracompact space, as well as every countable-dimensional metrizable
space, is a $C$-space \cite{addis-gresham:78}, but there exists a
compact metric $C$-space which is not countable-dimensional
\cite{pol:81}.\medskip

A set-valued mapping $\Phi:X\sto Y$ is called \emph{lower locally
  constant} \cite{gutev:05} if the set
$\{x\in X:K\subset \Phi(x)\}$ is open in $X$, for every compact
subset $K\subset Y$.  This property appeared in a paper of Uspenskij
\cite{uspenskij:98}; later on, it was used by some authors (see, for
instance, \cite{chigogidze-valov:00a,valov:00}) under the name
``strongly l.s.c.'', while in papers of other authors strongly l.s.c.\
was already used for a different property of set-valued mappings (see,
for instance, \cite{gutev:95e}).  Clearly, every lower locally
constant mapping is l.s.c.\ but the converse fails in general and
counterexamples abound. In fact, if we consider a single-valued map
$f:X\to Y$ as a set-valued one, then $f$ is l.s.c.\ if and only if it
is continuous, while $f$ will be lower locally constant if and only if
it is locally constant. Thus, the term ``lower locally constant''
provides some natural analogy with the single-valued case.\medskip

The following theorem was obtained by Uspenskij \cite{uspenskij:98}.

\begin{theorem}[\cite{uspenskij:98}]
  \label{theorem-shsa-v26:1}
  Let $X$ be a paracompact $C$-space, $Y$ be a topological space, and
  $\Phi : X\sto Y$ be a lower locally constant mapping with
  aspherical values. Then $\Phi$ has a continuous selection.
\end{theorem}

In view of the method for proving Theorem \ref{theorem-shsa-v21:6}, the
above selection theorem can be considered as a ``selection
approximation property'' for the case of l.s.c.\ mappings, see
Proposition \ref{proposition-shsa-v28:4}.  This brings the following
natural question.

\begin{question}
  \label{question-shsa-v26:1}
  Let $X$ be a paracompact $C$-space, $(Y,\rho)$ be a complete metric
  space and $\Phi : X\to \mathscr{F}(Y)$ be an l.s.c.\ mapping. What
  conditions on the family $\{\Phi(x): x\in X\}$ will guarantee the
  existence of a continuous selection for $\Phi$? What if
  $\{\Phi(x): x\in X\}$ is uniformly equi-$LC^\omega$ and each
  $\Phi(x)$, $x\in X$, is $C^\omega$?
\end{question}

Turning to stronger properties in the setting of this question, let us
recall that a metrizable space $S$ is an \emph{Absolute Extensor} for
the metrizable spaces, or shortly an \emph{AE}, if every continuous map
from a closed subset $A$ of a metrizable space $X$ into $S$ can be
extended to a continuous map of $X$ into $S$. If every continuous map
from a closed subset $A$ of a metrizable space $X$ into $S$ can be
extended to a continuous map in $S$ over some neighbourhood of $A$ in
$X$, then $S$ is called an \emph{ANE}.  A collection $\mathscr{S}$ of
subsets of a metric space $(Y,\rho)$ is called \emph{uniformly
  equi-$LAE$} (\emph{Local Absolute Extensor}) \cite{MR0328858} if for
every $\varepsilon>0$ there exists $\delta>0$ such that if $g$ is a
continuous map from a closed subset $A$ of a metrizable space $X$ into
any $S\in \mathscr{S}$ with $\diam(g(A))<\delta$, then it has a
continuous extension $f: X\to S$ with $\diam(f(X))<\varepsilon$.  A
collection $\mathscr{S}$ of subsets of a metric space $(Y,\rho)$ is
called \emph{uniformly equi-$LC$} (\emph{Locally Contractible})
\cite{MR0328858} if for every $\varepsilon>0$ there exists $\delta>0$
such that if $p\in S\in \mathscr{S}$, then
$\mathbf{O}_\delta(p)\cap S$ is contractible over a subset of $S$ of
diameter $<\varepsilon$. Clearly, $S$ is an AE implies that $S$ is
$C^\omega$, and each uniformly equi-$LAE$ family $\mathscr{S}$ is is
equi-$LC^\omega$. Moreover, the following was shown by Pixley
\cite[Theorem 3.1]{MR0328858}.

\begin{theorem}[\cite{MR0328858}]
  \label{theorem-shsa-v26:3}
  For a collection $\mathscr{S}$ of subsets of a metric space
  $(Y,\rho)$, the following conditions are equivalent\textup{:}
  \begin{enumerate}
  \item The collection $\mathscr{S}$ is uniformly equi-$LAE$.
  \item Each $S\in\mathscr{S}$ is an ANE and $\mathscr{S}$ is
    uniformly equi-$LC$.
  \end{enumerate}
\end{theorem}

Regarding the role of such properties in the selection problem for
l.s.c.\ mappings, the following example was given by Pixley
\cite[Theorem 1.1]{MR0328858}.

\begin{theorem}[\cite{MR0328858}]
  \label{theorem-shsa-v26:2}
  Let $\mathbf{Q}=[0,1]^\omega$ be the Hilbert cube. Then there exists
  an l.s.c.\ mapping $\Phi:\mathbf{Q}\to \mathscr{F}(\mathbf{Q})$
  such that
\begin{enumerate}[label=\upshape{(\roman*)}]
\item The collection $\{\Phi(x):x\in \mathbf{Q}\}$ is uniformly
  equi-$LC$,
\item Each $\Phi(x)$, $x\in \mathbf{Q}$, is either a point, or
  homeomorphic to a $k$-cell \textup{(}for some $k\geq 1$\textup{)},
  or homeomorphic to $\mathbf{Q}$,
\item There is no continuous selection for $\Phi$. In fact,
  $\Phi$ has no the SNEP at some singleton of $\mathbf{Q}$.
\end{enumerate}
\end{theorem}

Uniformly equi-$LAE$ families are another candidate to try Question
\ref{question-shsa-v26:1}.

\begin{question}
  \label{question-shsa-v26:2}
  Let $X$ be a paracompact $C$-space, $(Y,\rho)$ be a complete metric
  space, and $\Phi : X\to \mathscr{F}(Y)$ be an l.s.c.\ mapping
  such that $\{\Phi(x): x\in X\}$ is uniformly equi-$LAE$ and each
  $\Phi(x)$, $x\in X$, is $AE$. Does there exist a continuous
  selection for $\Phi$?
\end{question}

Let $|\Sigma|$ be the geometric realisation of a simplicial complex
$\Sigma$, and $\mathbf{V}_\Sigma$ be its vertices. As a topological
space, we will consider $|\Sigma|$ endowed with the \emph{Whitehead
  topology}. In this topology, a subset $U\subset |\Sigma|$ is open if
and only if $U\cap |\sigma|$ is open in $|\sigma|$, for every simplex
$\sigma\in\Sigma$. Motivated by the Lefschetz characterisation
\cite{MR0007094} of compact metrizable absolute neighbourhood
retracts, Pixley considered the following stronger condition. A
collection $\mathscr{S}$ of subsets of a metric space $(Y,\rho)$ is
called \emph{uniformly equi-$(L)$} if for each $\varepsilon>0$, there
exists $\delta>0$ with the property that for each simplicial complex
$\Sigma$ and subcomplex $K\subset \Sigma$ with
$\mathbf{V}_\Sigma\subset K$, each continuous $g:|K|\to S$ with
$\diam(g(|K\cap \sigma|))<\delta$ for every $\sigma\in \Sigma$, can be
extended to a continuous $f:|\Sigma|\to S$ such that
$\diam(f(|\sigma|))<\varepsilon$, $\sigma\in \Sigma$. Pixley remarked
that each uniformly equi-$(L)$ family is uniformly equi-$LAE$ based on
a result of Lefschetz \cite[(6.6)]{MR0007094} and the extension
theorem of Dugundji \cite[Theorem 3.1]{dugundji:51}, but the converse
is not true due to Theorem \ref{theorem-shsa-v26:2}. Thus, he stated
implicitly the following question.

 \begin{question}[\cite{MR0328858}]
   \label{question-shsa-v26:3}
   Let $X$ be a paracompact space, $(Y,\rho)$ be a complete metric
   space, and $\Phi:X\to \mathscr{F}(Y)$ be an l.s.c.\ mapping such
   that $\{\Phi(x): x\in X\}$ is uniformly equi-$(L)$ and each
   $\Phi(x)$, $x\in X$, is an AE. Does there exist a continuous
   selection for $\Phi$? What if $\Phi$ is $\rho$-continuous?
 \end{question}

\subsection{Selections Avoiding Sets}
 
For spaces $S$ and $Y$, we use $C(S,Y)$ to denote the set of all
continuous maps from $S$ to $Y$ endowed with the compact-open
topology. In fact, $C(S,Y)$ will be used in the case of a compact $S$
and a metrizable $Y$, where this topology is the uniform topology on
$C(S,Y)$ generated by any compatible metric on $Y$. Finally, let
$\mathbf{Q}=[0,1]^\omega$ be the Hilbert cube; $\mathbf{I}^n=[0,1]^n$
be the $n$-cube, and $\mathbf{I}^0$ to be the $0$-cube (i.e.\ a
singleton).\medskip

A closed set $A$ of a metrizable space $Y$ is said to be a
\emph{$Z_n$-set in $Y$}, where $n\geq 0$, if
$C(\mathbf{I}^n,Y\backslash A)$ is dense in $C(\mathbf{I}^n,Y)$, and
$A$ is called a \emph{$Z$-set in $Y$} if $C(\mathbf{Q},Y\setminus A)$
is dense in $C(\mathbf{Q},Y)$, see \cite{chigogidze:96} and
\cite{torunczyk:78}. The collection of all $Z_n$-sets ($Z$-sets) in
$Y$ will be denoted by $\mathscr{Z}_{n}(Y)$ (respectively,
$\mathscr{Z}(Y)$). Let us remark that the elements of
$\mathscr{Z}_0(Y)$ are precisely the closed nowhere dense subsets of
$Y$. Moreover, it is well known that
$\mathscr{Z}(Y)=\bigcap_{n\geq0}\mathscr{Z}_n(Y)$, see e.g.\
\cite[Proposition 2.1 in \S 2 of Chapter V]{bessage-pelczynski:75}. In
view of this, the $Z$-sets are sometimes called
$Z_\infty$-sets. Finally, let us also remark that a closed subset
$A\subset\mathbf{Q}$ is a $Z$-set in $\mathbf{Q}$ iff the identity map
of $\mathbf{Q}$ can be uniformly approximated by continuous self-maps
of $\mathbf{Q}$ with range entirely contained $\mathbf{Q}\setminus A$,
see e.g.\ \cite[Lemma 6.2.3]{mill:89}.\medskip

The following characterisation of $C$-spaces was obtained by Uspenskij
\cite[Theorem 1.4]{uspenskij:98}.

\begin{theorem}[\cite{uspenskij:98}]
  \label{theorem-shsa-v26:4}
  A compact space $X$ is a $C$-space if and only if for every
  continuous mapping $\Psi : X\to \mathscr{Z}(\mathbf{Q})$ there
  exists a continuous map $f : X\to \mathbf{Q}$ such that $f(x)\notin
  \Psi(x)$,  for all $x\in X$.
\end{theorem}

Subsequently, Theorem \ref{theorem-shsa-v26:4} was extended to all
paracompact $C$-spaces in \cite[Theorem 1.1]{gutev-valov:00}.

\begin{theorem}[\cite{gutev-valov:00}]
  \label{theorem-shsa-v26:5}
  A paracompact space $X$ is a $C$-space if and only if whenever $E$
  is a Banach space, $\Phi:X\to \mathscr{F}_\mathbf{c}(E)$ is l.s.c.\
  and ${\Psi_n:X\to \mathscr{F}(E)}$, $n\in \N$, are closed-graph
  mappings with $\Psi_n(x)\cap\Phi(x)\in \mathscr{Z}(\Phi(x))$, for
  all $x\in X$ and $n\in \N$, there exists a continuous map $f:X\to E$
  with $f(x)\in \Phi(x)\setminus\bigcup_{n\in\N}\Psi_n(x)$, for
  every $x\in X$.
\end{theorem}

Every continuous mapping $\Psi:X\to \mathscr{Z}(\mathbf{Q})$ has a
closed graph. Hence, by taking $\Phi(x)=\mathbf{Q}$, $x\in X$, Theorem
\ref{theorem-shsa-v26:4} follows from Theorem
\ref{theorem-shsa-v26:5}. The interested reader is referred to
\cite{valov:00}, where finite $C$-spaces were characterised in a
similar manner; also to \cite{gutev-valov:01a}, where a natural
finite-dimensional version of Theorem \ref{theorem-shsa-v26:5} (in
terms of $Z_n$-sets) was obtained.

\subsection{Selections and Convex $G_\delta$-Sets}

If $Y$ is completely metrizable and an absolute extensor for the
metrizable spaces, then it is also an absolute extensor for the
collectionwise normal spaces \cite{dowker:52}. In particular, this is
true for every convex $G_\delta$-subset $Y$ of a Banach space
$E$. Indeed, $Y$ is an absolute extensor for the metrizable spaces
being convex (by Dugundji's extension theorem \cite{dugundji:51}), and
is also completely metrizable being a $G_\delta$-subset of a complete
metric space. Motivated by this, the following problem was stated by
Michael in \cite[Problem 396]{michael:90}, it became known as
\emph{Michael's $G_\delta$-problem}.

\begin{question}[\cite{michael:90}]
  \label{question-shsa-v28:1}
  Let $E$ be a Banach space, $Y\subset E$ be a convex
  $G_\delta$-subset of $E$, $X$ be a paracompact space and
  $\Phi:X\to \mathscr{F}_\mathbf{c}(Y)$ be an l.s.c.\ mapping. Does
  $\Phi$ have a continuous selection?
\end{question}

In general, the answer to this question is in the negative due to a
counterexample constructed by Filippov \cite{filippov:04,filippov:05},
see also \cite{Repovvs2007}. However, Question
\ref{question-shsa-v28:1} was also resolved in the affirmative in a
number of partial cases. As Michael remarked in \cite[Remark
3.11]{michael:90}, the answer is ``Yes'' if
$\overline{\conv(K)}\subset \Phi(x)$ for every compact subset
$K\subset\Phi(x)$ and $x\in X$. For instance, this is true if $Y$ is a
countable intersection of open convex sets, or $\dim \Phi(x)<\infty$,
for all $x\in X$. Various related observations for this special case
can be found in \cite{MR2848743,MR2406392}.  Another remark made by
Michael is that the answer is ``Yes'' provided $\dim(X)<\infty$
\cite[Remark 3.6]{michael:90}.\medskip

All these special cases were generalised in \cite[Theorem
1.1]{gutev:94} by showing that in the setting of Question
\ref{question-shsa-v28:1}, continuous selections are equivalent to the
existence of convex-valued usco selections (i.e.\
convex-compact-valued u.s.c.\ selections).

\begin{theorem}[\cite{gutev:94}]
  \label{theorem-shsa-v26:7}
  For a paracompact space $X$ and a $G_\delta$-subset $Y$ of a Banach
  space $E$, the following are equivalent\textup{:}
  \begin{enumerate}
  \item Every l.s.c.\ $\Phi:X\to \mathscr{F}_\mathbf{c}(Y)$ has a
    continuous selection.
  \item Every l.s.c.\ $\Phi:X\to \mathscr{F}_\mathbf{c}(Y)$ has an
    usco convex-valued selection. 
  \end{enumerate}
\end{theorem}

Furthermore, Theorem \ref{theorem-shsa-v26:7} implies that the answer
is ``Yes'' when $X$ is a countable-dimensional metrizable space
\cite[Corollary 1.2]{gutev:94}, or when it is a strongly
countable-dimensional paracompact space \cite[Corollary
1.3]{gutev:94}. \medskip

Finally, let us explicitly remark that Question
\ref{question-shsa-v28:1} was also resolved in the affirmative in the
realm of $C$-spaces \cite[Theorem 4.4]{gutev-valov:00}.

\begin{theorem}[\cite{gutev-valov:00}]
  \label{theorem-shsa-v13:1}
  Let $X$ be a paracompact $C$-space, $E$ be a Banach space and $Y$
  be a $G_{\delta}$-subset of $E$. Then every l.s.c.\ mapping
  $\Phi:X\to \mathscr{F}_\mathbf{c}(Y)$ has a continuous selection.
\end{theorem}

This result is, in fact, based on Theorem \ref{theorem-shsa-v26:5} and
the following property of convex $G_\delta$-sets in Banach space, see
\cite[Lemma 4.3]{gutev-valov:00}.

\begin{lemma}[\cite{gutev-valov:00}]
  \label{lemma-shsa-v26:1}
  If $E$ is a Banach space, $H\subset E$ is a convex $G_\delta$-subset
  of $E$ and $F\subset E$ is a closed set with
  $F\cap H\in\mathscr{Z}(H)$, then
  $F\cap \overline{H}\in\mathscr{Z}\left(\overline{H}\right)$.
\end{lemma}

Motivated by this, the following question was posed by Repov\v{s} and
Semenov in \cite[Problem 1.6]{MR1970007}, and subsequently in
\cite[Problem 2.6]{MR3205497}.

\begin{question}[\cite{MR1970007,MR3205497}]
  \label{question-shsa-v26:4}
  Let $X$ be a paracompact space such that for every 
  $G_\delta$-subset $Y\subset E$ of a Banach space $E$, every l.s.c.\
  $\Phi:X\to \mathscr{F}_\mathbf{c}(Y)$ has a continuous
  selection. Does this imply that $X$ is a $C$-space?
\end{question}

Question \ref{question-shsa-v26:4} was resolved by Karassev
\cite[Theorem 4.6]{karassev:08} for the case of weakly
infinite-dimensional compact spaces.

\begin{theorem}[\cite{karassev:08}]
  \label{theorem-shsa-v26:6}
  Let $X$ be a compact space such that for every $G_\delta$-subset
  $Y\subset E$ of a Banach space $E$, every l.s.c.\ mapping
  $\Phi:X\to \mathscr{F}_\mathbf{c}(Y)$ has a continuous
  selection. Then $X$ is weakly infinite-dimensional.
\end{theorem}

As remarked in \cite{MR3205497}, perhaps what is also interesting is the
implicit relation of Theorem \ref{theorem-shsa-v26:6} to one of the
main problems in infinite dimension theory of whether every weakly
infinite-dimensional compact metric space has property $C$. In the
realm of compact spaces, there are various characterisations of weak
infinite-dimensionality. In case of Theorem \ref{theorem-shsa-v26:6},
Karassev used the following property, compare with
Theorem \ref{theorem-shsa-v26:4}.

\begin{theorem}[\cite{MR1194494}]
  \label{theorem-shsa-v28:1}
  A compact space $X$ is weakly infinite-dimensional if and only if
  for every continuous map $g :X\to \mathbf{Q}$ in the Hilbert cube
  $\mathbf{Q}$, there exists a continuous map $f:X\to \mathbf{Q}$ with
  $f (x)\neq g(x)$, for all $x\in X$.
\end{theorem}

Another aspect of Question \ref{question-shsa-v28:1} was considered in
\cite{MR2643824}. Namely, in view of the relationship between
selections and extension (see Theorem \ref{theorem-res-pro-v1:4}), the
following question was posed in \cite[Question 2]{MR2643824}.

\begin{question}[\cite{MR2643824}]
  \label{question-shsa-v12:1}
  Let $X$ be a collectionwise normal space, $E$ be a Banach space,
  $Y\subset E$ be a convex $G_\delta$-subset of $E$ and
  $\Phi:X\to \mathscr{C}'_\mathbf{c}(Y)$ be an l.s.c.\ mapping. Does
  $\Phi$ have a continuous selection?
\end{question}

This question is not only similar to Question
\ref{question-shsa-v28:1}, but most of the affirmative solutions of
Question \ref{question-shsa-v28:1} remain valid for it as
well. Indeed, if $\overline{\conv(K)}\subset \Phi(x)$ for every
compact subset $K\subset \Phi(x)$ and $x\in X$, by a result of
\cite{choban-valov:75}, $\Phi$ has an l.s.c.\ selection
$\varphi:X\to \mathscr{C}_\mathbf{c}(Y)$. Hence, $\Phi$ also has a
continuous selection because, by Theorem \ref{theorem-res-pro-v1:2},
so does $\varphi$. If $X$ is finite-dimensional, the answer is also
``Yes'', and follows directly from a selection theorem in
\cite{gutev:86}. The answer to Question \ref{question-shsa-v12:1} is
also ``Yes'' if $X$ is strongly countable-dimensional. In this case,
the mapping ${\Phi:X\to \mathscr{C}'_\mathbf{c}(Y)}$ has an l.s.c.\
weak-factorisation $(Z,g,\varphi)$ with $Z$ being a strongly
countable-dimensional space, see, for instance, the proof of
\cite[Theorem 5.3]{nedev:80}. Then just as in the proof of Proposition
\ref{proposition-shsa-v17:1}, $\Phi$ has a continuous selection
because so does the mapping
$\overline{\conv[\varphi]}^Y:Z\to \mathscr{F}_\mathbf{c}(Y)$, where
the closure is in $Y$. \medskip

In contrast to Theorem \ref{theorem-shsa-v13:1}, Question
\ref{question-shsa-v12:1} is still open for collectionwise normal
$C$-spaces.

\begin{question}
  \label{question-shsa-v13:3}  
  Let $X$ be a collectionwise normal $C$-space, $E$ be a Banach space,
  $Y$ be a convex $G_{\delta}$-subset of $E$ and
  $\Phi:X\to \mathscr{C}'_\mathbf{c}(Y)$ be an l.s.c.\ mapping.  Does
  there exist a continuous selection for $\Phi$?
\end{question}

As mentioned before, a countably paracompact normal $C$-space is
paracompact, in which case the answer is ``Yes'', by Theorem
\ref{theorem-shsa-v13:1}. Hence, Question \ref{question-shsa-v13:3} is
for collectionwise normal spaces which are not countably
paracompact. \medskip

Every countable-dimensional metrizable space has property $C$, and
more generally, every countable-dimensional hereditarily paracompact
space is a $C$-space \cite[Corollary 2.10]{addis-gresham:78}. However,
it seems it is still an open question if every countable-dimensional
paracompact space is a $C$-space \cite[Question 1]{MR2352366}. This
brings the following related question in terms of selections and
$G_\delta$-sets of Banach spaces.

\begin{question}
  \label{question-shsa-v26:5}
  Let $X$ be a countable-dimensional paracompact space, $Y\subset E$
  be a (convex) $G_\delta$-subset of a Banach space and $\Phi:X\to
  \mathscr{F}_\mathbf{c}(Y)$ be an l.s.c.\ mapping. Does there exist
  a continuous selection for $\Phi$?
\end{question}

\subsection{Disjoint Sections and Selections}
  
There is a further relationship between $C$-spaces and the property of
weakly infinite-dimensional spaces used by Karassev
\cite{karassev:08}, see Theorem \ref{theorem-shsa-v28:1}. A pair of
set-valued mappings $\varphi,\psi:X\sto Y$ are called \emph{disjoint}
if $\varphi(x)\cap \psi(x)=\emptyset$, for every $x\in X$. It was
shown by Dranishnikov \cite[Theorem 1]{dranishnikov:90}, see also
\cite{dranishnikov:88}, that the fibration
\[
  \eta=\prod_{i=0}^\infty \nu_{2^i}:\prod_{i=0}^\infty \s^{(2^i)}\to
  \mathbb{RP}^{(2^i)}
\]
does not accept two disjoint usco sections, where
$\nu_k:\s^k\to \mathbb{RP}^k$ is a $2$-fold covering map of the
$k$-sphere onto the real projective $k$-space. Here, an usco section
for $\eta$ is meant an usco set-valued selection for the inverse
set-valued mapping $\eta^{-1}$. Since $\eta$ is open, its inverse is
l.s.c., so this is an example of an l.s.c.\ infinite-valued mapping
which doesn't admit a pair of disjoint usco selections. On the other
hand, the following was shown in \cite[Corollary 4.4]{MR2545446}.

\begin{theorem}[\cite{MR2545446}]
  \label{theorem-shsa-v26:9}
  Let $X$ be a paracompact $C$-space, $Y$ be completely metrizable
  and $\Phi:X\to \mathscr{F}(Y)$ be an l.s.c.\ mapping such that each
  $\Phi(x)$, $x\in X$, is infinite. Then $\Phi$ has a pair of disjoint
  usco selections. 
\end{theorem}

Regarding disjoint usco selections, the following part of
\cite[Problem 1516]{karassev-tuncali-valov:07} is still open.

\begin{question}[\cite{karassev-tuncali-valov:07}]
  \label{question-shsa-v26:7}
  Let $X$ be a metrizable space such that for every metrizable space
  $Y$, any l.s.c.\ mapping $\Phi : X\to \mathscr{C}(Y)$ with perfect
  point-images $\Phi(x)$, $x\in X$, admits disjoint
  u.s.c. selections. Is it true that $X$ is a $C$-space?
\end{question}

A map $f:X\to E$ is said to \emph{avoid} some set $Z\subset E$, if
$f(x)\notin Z$ for all $x\in X$. In case $E$ is a linear space and
$Z=\{\mathbf{0}\}$ is the singleton of the origin of $E$, the map is
simply called \emph{$\mathbf{0}$-avoiding}. In \cite{michael:88a},
Michael considered the following natural problem for
$\mathbf{0}$-avoiding selections.

\begin{question}[\cite{michael:88a}]
  \label{question-shsa-vgg:1}
  Let $X$ be a paracompact space, $E$ be a Banach space and
  $\Phi:X\to \mathscr{F}_\mathbf{c}(E)$ be an l.s.c.\ mapping such
  that $\Phi(x)\neq \{\mathbf{0}\}$, for every $x\in X$. Under what
  conditions, does $\Phi$ have a continuous $\mathbf{0}$-avoiding
  selection?
\end{question}

He remarked that in setting of selection theorems such as Theorem
\ref{theorem-res-pro-v9:1}, the constructed continuous selections
cannot be chosen to be $\mathbf{0}$-avoiding (even when
$\Phi(x)\neq \{\mathbf{0}\}$ for all $x\in X$), and provided several
examples, see \cite[Examples 10.1 and 10.2]{michael:88a}. In case of
dimension restrictions on $X$, or strengthening the continuity of
$\Phi$, he obtained the following theorems.

\begin{theorem}[\cite{michael:88a}]
  \label{theorem-shsa-v27:1}
  Let $X$ be a paracompact space, $E$ be a Banach space and
  $\Phi :X \to \mathscr{F}_\mathbf{c}(E)$ be an l.s.c.\ mapping such
  that $\dim(X)<\dim\Phi(x)$, whenever $x\in X$ and
  $\mathbf{0}\in\Phi(x)$. Then $\Phi$ has a continuous
  $\mathbf{0}$-avoiding selection.
\end{theorem}

\begin{theorem}[\cite{michael:88a}]
  \label{theorem-shsa-v27:2}
  Let $X$ be an arbitrary space, $E$ be a Banach space and
  ${\Phi: X \to \mathscr{F}_\mathbf{c}(E)}$ be a norm-continuous
  mapping with $\dim\Phi(x)=\infty$, for all ${x\in X}$. If
  $\frac{u}{\|u\|}\in \Phi(x)$, whenever $\mathbf{0}\neq u\in \Phi(x)$
  and $x\in X$, then $\Phi$ has a continuous $\mathbf{0}$-avoiding
  selection.
\end{theorem}

Regarding the proper place of Theorem \ref{theorem-shsa-v27:2},
Michael stated the following question in \cite[Problem
395]{michael:90}. 

\begin{question}[\cite{michael:90}]
  \label{question-shsa-v27:1}
  Let $X$ be a paracompact space, $E$ be an infinite-dimensional
  Banach space and $\Phi:X\to \mathscr{F}(E)$ be an l.s.c.\ mapping
  with each $\Phi(x)$ a linear subspace of deficiency one (or of
  finite deficiency) in $E$.  Must $\Phi$ have a continuous
  $\mathbf{0}$-avoiding selection?
\end{question}

According to \cite[Remark 3.7]{michael:90}, the answer to this
question is ``No'' if it is only assumed that each $\Phi(x)$ is
infinite-dimensional. This follows from the mentioned example of
Dranishnikov \cite{dranishnikov:88} and a similar example of
Toru\'nczyk and West \cite{torunczyk1989}. Here are some further
remarks regarding Question \ref{question-shsa-v27:1}.

\begin{corollary}
  \label{corollary-shsa-v27:1}
  Let $X$ be a paracompact $C$-space, $E$ be a Banach space and
  $\Phi:X\to \mathscr{F}_\mathbf{c}(E)$ be an l.s.c.\ mapping with 
  $\dim \Phi(x)=\infty$, for every $x\in X$. Then 
  $\Phi$ has a continuous $\mathbf{0}$-avoiding selection.  
\end{corollary}

\begin{proof}
  Whenever $x\in X$, the singleton $\{\mathbf{0}\}$ is a $Z$-set in
  $\Phi(x)$ because $\Phi(x)$ is infinite-dimensional; equivalently,
  $\{\mathbf{0}\}$ a $Z_n$-set in $\Phi(x)$ for every $n\geq 0$ (for
  instance, one can apply Theorem \ref{theorem-shsa-v27:1} and the
  definition of a $Z_n$-set). Then the existence of a continuous
  $\mathbf{0}$-avoiding selection for $\Phi$ follows from Theorem
  \ref{theorem-shsa-v26:5} by taking $\Psi_n(x)=\{\mathbf{0}\}$ for
  all $x\in X$ and $n\in \N$.
\end{proof}

As Michael emphasised in \cite{michael:88a}, the benefit of Theorems
\ref{theorem-shsa-v27:1} and \ref{theorem-shsa-v27:2} is that they
actually show the existence of continuous selections avoiding given
continuous maps. Namely, suppose that $g : X\to E$ is a continuous map
in a Banach space, and $\Phi:X\to \mathscr{F}_\mathbf{c}(E)$ is an
l.s.c.\ mapping. Then one can consider the l.s.c.\ mapping
$\Phi-g:X\to \mathscr{F}_\mathbf{c}(E)$, defined by
$[\Phi-g](x)=\Phi(x)-g(x)$, $x\in X$. If $\Phi-g$ has a continuous
$\mathbf{0}$-avoiding selection $f:X\to E$, then $h=f+g:X\to E$ is a
continuous selection for $\Phi$ with $h(x)\neq g(x)$ for all $x\in
X$. Based on this and the characterisation of weakly
infinite-dimensional compact spaces in Theorem
\ref{theorem-shsa-v28:1}, we also have the following observation.

\begin{proposition}
  \label{proposition-nach space}
  Let $X$ be a compact space such that for every
  \textup{(}separable\textup{)} Banach space $E$, every l.s.c.\
  mapping $\Phi:X\to \mathscr{F}_\mathbf{c}(E)$ with
  $\dim\Phi(x)=\infty$, for all $x\in X$, has a continuous
  $\mathbf{0}$-avoiding selection. Then $X$ is weakly
  infinite-dimensional.
\end{proposition}

\begin{proof}
  Consider the Hilbert cube $\mathbf{Q}$ as a compact convex subset of
  $E=\ell_2=\ell_2(\N)$, or, alternatively, of
  $\ell_1=\ell_1(\N)$. Take a continuous map $g:X\to \mathbf{Q}$, and
  next define an l.s.c.\ mapping $\Phi:X\to \mathscr{F}_\mathbf{c}(E)$
  by $\Phi(x)=\mathbf{Q}-g(x)$, $x\in X$, see \cite[Example
  10.2]{michael:88a}. Then $\dim\Phi(x)=\infty$, for all $x\in X$, and
  by condition, $\Phi$ has a continuous $\mathbf{0}$-avoiding
  selection $h:X\to E$. As discussed above, this implies that
  $f=h+g:X\to \mathbf{Q}$ is a continuous map with $f(x)\neq g(x)$,
  for every $x\in X$. According to Theorem \ref{theorem-shsa-v28:1},
  $X$ is weakly infinite-dimensional.
\end{proof}

\subsection*{Acknowledgement}
The author is very much indebted to the referee who contributed an
alternative proof of Theorem \ref{theorem-res-pro-v2:1}, and whose
helpful and essential corrections greatly improved the final version
of this paper.

\providecommand{\bysame}{\leavevmode\hbox to3em{\hrulefill}\thinspace}


\end{document}